\documentclass{amsart}
\usepackage[usenames,dvipsnames]{color}
\usepackage[dvips]{graphicx}
\usepackage{dsfont}
\usepackage[arc,all]{xy}
\usepackage{multirow, colortbl}

\renewcommand{\epsilon}{\varepsilon}
\newcommand{\dd}{\mathrm{d}}

\newcommand{\w}{{\tt w}}

\newcommand{\fB}{\mathfrak{B}}

\definecolor{GGray}{gray}{0.65}

\newcommand{\N}{{\mathbb N}}
\newcommand{\Z}{{\mathbb Z}}
\newcommand{\Q}{{\mathbb Q}}
\newcommand{\R}{{\mathbb R}}
\newcommand{\C}{{\mathbb C}}
\newcommand{\T}{{\mathbb T}}

\newcommand{\Vol}{\operatorname{Vol}}
\newcommand{\Covol}{\operatorname{Covol}}

\newcommand{\Card}{\operatorname{Card}}
\newcommand{\Cov}{\operatorname{Cov}}
\newcommand{\eff}{\operatorname{eff}}

\newcommand{\cH}{{\mathcal H}}
\newcommand{\cQ}{{\mathcal Q}}
\newcommand{\cQt}{{\tilde{\mathcal Q}}}

\newtheorem{lem}{Lemma}
\newtheorem{prop}{Proposition}
\newtheorem{theo}{Theorem}
\newtheorem{defi}{Definition}
\theoremstyle{remark}
\newtheorem{rem}{Remark}
\newtheorem{Example}{Example}
\newtheorem{conv}{Convention}

\def \equi#1{\mathrel{\mathop{\kern 0pt\sim}\limits_{#1}}} 

\begin{document}

\date{}

\title[Volumes of strata]
{Volumes of strata of moduli spaces of quadratic differentials: getting explicit values}

\author[E.~Goujard]{Elise Goujard}
\address[Elise Goujard]{Max Planck Institute for Mathematics, Bonn, Germany}
\email{elise.goujard@gmail.com}

\begin{abstract}
The volumes of strata of Abelian or quadratic differentials play an important
role in the study of dynamics on flat surfaces, related to dynamics in
polygonal billiards. This article applies all known approaches to compute volumes in the quadratic case and provides explicit values of volumes of the strata
of meromorphic quadratic differentials with at most simple poles
in all dimensions up to 11.
\end{abstract}

\maketitle

\section{Introduction}

\subsection{Flat surfaces, quadratic differentials, moduli spaces and volumes of strata}
A meromorphic quadratic differential $q$ with at most simple poles on a Riemann surface $S$ of genus $g$ defines a flat metric on $S$ with conical singularities. If $q$ is \emph{not} the global square of a holomorphic 1-form on $S$ (also called  Abelian differential), the metric has a non-trivial linear holonomy group, and in this case $(S,q)$ is called a {\it half-translation} surface. In this paper, if it is not precised, we consider only quadratic differentials satisfying the previous condition. If $\alpha=\{\alpha_1, \dots, \alpha_n\}\subset\{-1\}\cup\N$ is a partition of $4g-4$, $\cQ(\alpha)$ denotes the moduli space of pairs $(S,q)$ as above, where $q$ has exactly $n$ singularities of orders given by $\alpha$. It is a {\it stratum} in the moduli space $\cQ_g$ of pairs $(S,q)$ with no additional constraints on $q$. Similarly if 
$\beta=\{\beta_1, \dots, \beta_m\}\subset\N$ is a partition of $2g-2$, $\cH(\beta)$ denotes the moduli space of Abelian differentials with zeros of degree $\beta$.

In the following we will refer to a half-translation surface $(S,q)$ simply as $S$.

 Any flat surface $(S,q)$ in $\cQ(\alpha)$ admits a canonical ramified double cover $\hat{S}\overset{p}{\rightarrow} S$ such that the induced quadratic differential on $\hat{S}$ is a global square of an Abelian differential, that is $p^* q= \omega^2$ and $(\hat S, \omega)\in \cH(\beta)$. Let ${\Sigma=\{P_1, \dots P_n\}}$ denote the singular points of the quadratic differential on $S$, and $\hat\Sigma=\{\hat P_1, \dots \hat P_N\}$ the singular points of the Abelian differential $\omega$ on $\hat S$. Note that the pre-images of poles $P_i$ are regular points of $\omega$ so do not appear in the list $\hat\Sigma$. The subspace $H^1_-(\hat S, \hat\Sigma; \C)$ antiinvariant with respect to the action of the hyperelliptic involution provides local coordinates in the stratum $\cQ(\alpha)$ in the neighborhood of $S$.
 
 \begin{conv}\label{convarea} Following \cite{AEZ} we denote by $\cQ_1(\alpha)$ the hypersurface in $\cQ(\alpha)$ of flat surfaces of area $1/2$ such that the area of the double cover is $1$.\end{conv}

The stratum $ \cQ(\alpha) $ is a complex orbifold of dimension $2g+n-2$, and it is equipped with a natural $ PSL(2, \R) $-invariant measure $ \mu $, called \mbox{Masur--Veech} measure, induced by the Lebesgue measure in period coordinates. This measure defines a measure $ \mu_1 $ on $ \cQ_1(\alpha) $ in the following way:
if $ E $ is a subset of $ \cQ_1(\alpha) $, we denote by $ C(E) $ the cone underneath $ E $ in the stratum $ \cQ(\alpha) $: \[C(E)=\{S\in\cQ(\alpha)\mbox{ s.t. } \exists r\in (0, 1), S=rS_1 \mbox{ with } S_1\in E\} \]
and we define 
\[\mu_1(E)=2d\cdot\mu(C(E)), \]
with $ d=\dim_\C\cQ(\alpha) $, that is, the measure $ \dd\mu $ disintegrates in ${\dd\mu=r^{2d-1}\dd r\dd\mu_1}$.
With this convention the volume of a stratum $\cQ(\alpha)$ is then given by:
\[\Vol\cQ_1(\alpha)=2d\Vol C(\cQ_1(\alpha)).\]

There are several possible choices for the normalization of $\mu$, two of them being commonly used: namely the choice of Athreya--Eskin--Zorich, described in \cite{AEZ} and recalled in \S~\ref{desvol}, and the choice of Eskin--Okounkov, described in \cite{EO2} and recalled on \S~\ref{convEO}. 

\subsection{Historical remarks}
In the case of Abelian differentials, volumes of strata with respect to the Masur--Veech measure were computed by Eskin and Okounkov (\cite{EO}), and by Kontsevich and Zorich in some low genus cases (\cite{Z}). The first authors used representation theory and modular forms, and their approach allowed them to prove the rationality of volumes which was conjectured by Kontsevich and Zorich, that is \[\Vol \cH_1(\beta)=r\cdot \pi^{2g},\; r\in\Q,\] where $g$ is the genus of the surfaces in the stratum $\cH(\beta)$. They also computed algorithmically the exact values of the volumes of strata up to genus 10.
Zorich used a combinatorial approach to compute explicitly the volumes of some strata in low genus.

Similar approaches were developed in the quadratic case. Eskin and Okounkov applied in \cite{EO2} similar methods as in the Abelian case, but this case presents many extra difficulties. Nevertheless, the rationality of volumes is still valid, that is \begin{equation}\label{eq:rat}\Vol\cQ_1(\alpha)=r\cdot \pi^{2g_{\mathrm{eff}}}, \; r\in\Q,\end{equation}
where $g_{\eff}=\hat g-g$ and $ \hat g $ is the genus of the double cover $ \hat S $ for $S\in\cQ(\alpha)$ (cf Lemma \ref{lem:rat}).

In the case of genus 0 surfaces, Athreya--Eskin--Zorich developed two parallel approaches that leaded to the explicit computation of volumes. The first one (\cite{AEZ2}) is combinatorial and is based on a formula of Kontsevich (\cite{K}). The second one develops the study of Siegel--Veech constants: they give a formula relating Siegel--Veech constants and volumes (based on the classification of configurations in \cite{MZ}, \cite{B}), and since the Siegel--Veech constants in genus 0 are known thanks to the Eskin--Kontsevich--Zorich formula (\cite{EKZ}), they deduce the volumes of strata for genus 0.

Independently Mirzakhani proved in \cite{Mi} a formula relating the volumes of the principal strata $\cQ(1^{4g-4})$ with the intersection pairings of tautological classes on moduli spaces of Riemann surfaces.

However, up to the present paper, none of the algorithm was implemented to produce explicit values of the rational numbers $r$ in (\ref{eq:rat}) for $g>0$.

The strata of moduli spaces of quadratic differentials may be disconnected \cite{L2}. Approximate values of volumes of connected components of strata of small dimension are computed in \cite{DGZZ}: note that for now this experimental method is the only one that provides numerical values of volumes of $\cQ^{\textrm{reg}}(9,-1)$ and $\cQ^{\textrm{irr}}(9,-1)$ separately.

\subsection{Motivation}

Values of volumes of moduli spaces of quadratic differentials arise in several problems related to billiards in polygons and interval exchange transformations.

In fact volumes are directly related to Siegel--Veech constants that give the asymptotic of the number of closed geodesics in flat surfaces, and the asymptotic of the number of closed trajectories in the corresponding polygonal billiards. 
 Furthermore, the Siegel--Veech constants are related to the sum of the Lyapunov exponents of the Hodge bundle along the Teichm\"uller geodesic flow over the stratum by a formula of Eskin--Kontsevich--Zorich \cite{EKZ}. These Lyapunov exponents give precious quantitative information about the dynamics in corresponding billiards: Using these exponents Delecroix, Hubert and Leli\`evre computed the diffusion rate in the wind-tree billiard \cite{Delecroix:Hubert:Lelievre} (see also \cite{DZ} for series of families of wind-tree billiards): they show that this diffusion rate is exactly $2/3$, so the dynamics differs radically from the dynamics of the random walk in the plane (diffusion rate $1/2$). 

The explicit formulas relating volumes of strata and Siegel--Veech constants are given in \cite{EMZ} for the Abelian case, and  in \cite{G} for the quadratic case, using the work of Masur--Zorich \cite{MZ}.

 The aim of this paper is to provide explicit exact values of volumes of strata, in order to get new explicit values of Siegel--Veech constants using \cite{G}, and consequently new sums of Lyapunov exponents. In particular this procedure can be applied to genus one surfaces (where there is only one Lyapunov exponent) to give new results in the vein of \cite{DZ} and \cite{AEZ}.

Furthermore this paper is the occasion to clean up all normalizations once for all, and to explain clearly how to pass from one to another, in order to make the values of volumes directly usable in any normalization.

\subsection{Non-varying strata}\label{sect:nonvar}

Evaluating volumes of strata is related with counting problems on half-translation surfaces. This link can be useful to compute volumes explicitly in some special cases.

For the strata of quadratic differentials in genus 0 , Athreya--Eskin--Zorich gave an explicit formula relating Siegel--Veech constants and volumes of strata in \cite{AEZ}.  The Eskin--Kontsech--Zorich formula (Theorem 2 of \cite{EKZ}) gives here the values of the Siegel--Veech constants for the strata. So they deduced the values of volumes.

In higher genera, the relation between Siegel--Veech constants and volumes is given in \cite{G}. But values of Siegel--Veech constants are not known in general, only numerical approximations can be obtained by simulating Lyapunov exponents and using the \cite{EKZ}-formula.

However for some special strata, called ``non-varying'', Chen and M\"oller showed in \cite{CM} that the sum of Lyapunov exponents is the same for the entire stratum and for all Teichm\"uller curves inside the stratum. For those strata they computed the constant sum of Lyapunov exponents, so we obtain the Siegel--Veech constants by applying \cite{EKZ}-formula.

All these strata have their boundary strata that are also either non-varying, or hyperelliptic and connected, or of genus 0, so we can use the recursions given by the relations \[c_{area}(\cQ(\alpha))=\cfrac{\mbox{Explicit polynomials in volumes of boundary strata}}{\Vol(\cQ_1(\alpha))}\]
given in \cite{G} to compute the exact values of their volumes.

This method is applied in \cite{G} for a bunch of examples. The results are coherent with those of the other sections.

\subsection{Structure of the paper}

We first recall the \cite{AEZ}-convention for the normalization of the volumes. In section \ref{sect:volhyp} we compute volumes of hyperelliptic components of strata using the known values of volumes in genus 0. Then we illustrate the combinatorial approach in genus different of 0 in section \ref{ssection:voldim5}.  Finally we follow the Eskin--Okounkov approach to compute all volumes up to dimension 10 (around 300 strata).  

Most sections of this paper are written with respect to the \cite{AEZ}-convention, the last section uses the \cite{EO2}-convention and gives the normalization factor between the two conventions. In Appendix \ref{tab} we give all volumes written in the \cite{AEZ}-convention up to dimension 10.

\subsection{Acknowledgments}
I wish to thank my advisor Anton~Zorich, for his guidance and support during the preparation of this paper. I am grateful to Alex~Eskin for his help concerning the computations of the last part. I thank Anton~Zorich, Vincent~Delecroix and Peter~Zograf for many numerical computations \cite{DGZZ} that were used to check the consistency of the computations of this paper. I would like to thank Martin~M\"oller for his help with the computer experiments: in particular the program for computing volumes is based on his program (with D. Chen) for Abelian differentials. I wish to thank Corentin~Boissy for pointing me out some symmetry issues, Julien~Courtiel for helpful discussions about combinatorial maps, Pascal~Hubert, Samuel~Leli\`evre, Martin~M\"oller and Rodolfo R\'ios-Zertuche for useful discussions about volumes. I thank the anonymous referee for careful reading of the manuscript and useful comments.
 I thank ANR GeoDyM for financial support. The computations of the last part of this paper were performed at the Max Planck Institute of Mathematics in Bonn.

\section{Description of the Athreya--Eskin--Zorich's convention on volumes}\label{desvol}

Choosing a normalization for the volume element on a strata $\cQ(\alpha)$ is equivalent to choose a lattice in the space $H^1_-(\hat S, \hat\Sigma; \C)$ which gives the local model of the stratum $\cQ(\alpha)$ around $S$. The volume is then normalized by declaring that the covolume of the lattice is 1.

\begin{conv}\label{convreseau} Following the convention of \cite{AEZ} we choose, as lattice in $H^1_-(\hat S, \hat\Sigma; \C)$ of covolume $1$, the subset of those linear forms which take values in $\Z \oplus i\Z$ on $H^-_1(\hat S, \hat\Sigma;\Z)$, that we will denote by $(H^-_1(\hat S, \hat\Sigma;\Z))^{*}_{\C}$.
\end{conv} 
In other words the local image of this lattice under the period map is  the lattice $(\Z\oplus i\Z)^{\dim_\C}$ in $\C^{\dim_\C}$ where $\dim_\C$ is the complex dimension of $\cQ(\alpha)$.

This convention implies that the non zero cycles in $H_1(S, \Sigma,\Z)$ (that is, those represented by saddle connections joining two distinct singularities or closed loops non homologous to zero) have half-integer holonomy, and the other ones (closed loops homologous to zero) have integer holonomy.

 We denote by $\Vol^{numb}\cQ(\alpha)$ the volume of the stratum $\cQ(\alpha)$ when the zeros and poles are numbered and by $\Vol^{unnumb}\cQ(\alpha)$ the volume of the stratum when they are not. We have the following relation: 
\begin{align}\label{eq:label}\Vol^{numb}\cQ_1(\alpha_1^{m_1}, \alpha_2^{m_2}, \dots, \alpha_s^{m_s})=&\frac{m_1!m_2!\dots m_s!}{\vert\Gamma(\alpha)\vert} \;\cdot\notag\\
&\Vol^{unnumb}\cQ_1(\alpha_1^{m_1}, \alpha_2^{m_2}, \dots, \alpha_s^{m_s}) 
\end{align}
 where $\Gamma(\alpha)$ denotes the group of symmetries of all surfaces in the stratum $\cQ(\alpha)$. By symmetry of a surface in $\cQ(\alpha)$ we mean an automorphism of the flat surface. It preserves the set of singularities, permuting the labels of the singularities of same order. In general, surfaces in a stratum do not share any symmetry: some of them can be very symmetric (orbifold locus), but most of them are not. In some special cases (as $\cQ(-1^4)$ or hyperelliptic components), some symmetries are common to all surfaces, defining the group $\Gamma(\alpha)$. We abuse notation by denoting by $|\Gamma(\alpha)|$ the cardinal of the group of permutations of the labels of the singularities induced by $\Gamma(\alpha)$.  

\begin{conv}\label{conv:label}We choose to label all zeros and poles. In other terms, we compute the volumes $\Vol^{numb}\cQ_1(\alpha)$ that we will simply denote by $\Vol\cQ_1(\alpha)$ in the rest of the paper.
\end{conv}

 Let $ \gamma $ be a saddle connection on $ S $. We denote by $\gamma'$ and $\gamma''$ its two lifts on $\hat S$. If $[\gamma]= 0$ in $H_1(S, \Sigma; \C)$, then $[\gamma']+[\gamma'']= 0$ in $H_1(\hat S, \hat \Sigma; \C)$, and in this case we define $[\hat \gamma]   := [\gamma']$. In the other case we have $[\gamma']+[\gamma'']\neq 0$ and we define $[\hat \gamma]:=[\gamma']-[\gamma'']$. We obtain
  an element of $H_1^-(\hat S, \hat\Sigma; \C)$.

For a primitive cycle $[\gamma]$ in $H_1(S, \Sigma,\Z)$, that is, a saddle connection joining distinct zeros or a closed cycle (absolute cycle), the lift $[\hat\gamma]$ is a primitive element of $H_1^-(\hat S, \hat\Sigma,\Z)$.

We recall the construction given in \cite{AEZ} of a basis of $H_1^-(\hat S, \hat\Sigma,\Z)$ from a basis of $H_1(S, \Sigma,\Z)$.  
\subsection{Basis of $H_1^-(\hat S, \hat\Sigma,\Z)$}(cf \cite{AEZ} $\S 3.1$)
Let $ k $ be the number of poles in $ \Sigma $, $ a $ the number of even zeros and $ b $ the number of odd zeros (of order $ \geq 1 $). Assume that the zeros are numbered in the following way: $P_1, \dots P_a$ are the even zeros, $P_{a+1}, \dots, P_{a+b}$ are the odd zeros and $P_{a+b+1}, \dots, P_n$ the poles, and take a simple oriented broken line $P_1, \dots P_{n-1}$. Take each saddle connection $\gamma_i$ represented by $[P_i, P_{i+1}]$ for $i$ going from $1$ to $n-2$, and a basis $\{\gamma_{n-1}, \dots, \gamma_{n+2g-2}\}$ of $H_1(S, \Z)$.

Then we have (cf \cite{AEZ} $\S 3.1$):
 \begin{lem}[Athreya-Eskin-Zorich]\label{lem:basis}The family $\{\hat \gamma_1,\dots ,\hat \gamma_{n+2g-2}\}$ is a basis of $H_1^-(\hat S, \hat\Sigma,\Z)$.\end{lem} 

This lemma will be useful for the computations of the next two sections.

\section{Using hyperellipticity}\label{sect:volhyp}

We begin with hyperelliptic components of strata: the values of their volumes are easier to compute since they are related to values of volumes in genus 0, that are computed in \cite{AEZ}.

\subsection{Volumes of hyperelliptic components of strata of quadratic differentials}\label{ssect:hyp}
The strata of the moduli spaces of quadratic differentials have one or two connected components: for genus $g\geq 5$ there are two components when the stratum contains a hyperelliptic component (cf \cite{L2}). For genus $g\leq 4$ some strata are hyperelliptic and connected (cf \cite{L1}): namely $\cQ(1^2, -1^2)$ and $\cQ(2, -1^2)$ in genus 1, $\cQ(1^4)$, $\cQ(2, 1^2)$, and $\cQ(2,2)$ in genus 2. For these strata and for hyperelliptic components of strata in higher genus the volume is easier to compute. 

\begin{prop} 
The volumes of hyperelliptic components of strata of quadratic differentials are given by the following formulas (in convention \cite{AEZ}):
\begin{itemize}
\item First type ($k_1\geq -1$ odd, $k_2\geq -1$ odd, $(k_1, k_2)\neq (-1, -1)$):

If $k_1\neq k_2$:
\begin{equation}\Vol\cQ_1^{hyp}(k_1^2, k_2^2)=\frac{2^{d}}{d!}\pi^{d}\frac{k_1!!}{(k_1+1)!!}\frac{k_2!!}{(k_2+1)!!}\label{eq:volQhyp1}\end{equation}
Otherwise:
\begin{equation}\Vol\cQ_1^{hyp}(k_1^4)=3\cdot\frac{ 2^{d}}{d!}\pi^{d}\left(\frac{k_1!!}{(k_1+1)!!}\right)^2\label{eq:volQhyp1part}\end{equation}

\item Second type ($k_1\geq -1$ odd, $k_2\geq 0$ even):
\begin{equation}\Vol\cQ_1^{hyp}(k_1^2, 2k_2+2)=\frac{2^{d}}{d!}\pi^{d-1}\frac{k_1!!}{(k_1+1)!!}\frac{k_2!!}{(k_2+1)!!}\label{eq:volQhyp2}\end{equation}

\item Third type ($k_1$, $k_2$ even):
\begin{equation}\Vol\cQ_1^{hyp}(2k_1+2, 2k_2+2)=\frac{2^{d+1}}{d!}\pi^{d-2}\frac{k_1!!}{(k_1+1)!!}\frac{k_2!!}{(k_2+1)!!}\label{eq:volQhyp3}\end{equation}
The same formula holds for $k_1=k_2$.
\end{itemize} In these formulas $d=k_1+k_2+4$ is the complex dimension of the strata.
\end{prop}

\begin{Example}
For the five strata that are connected and hyperelliptic we obtain:
\begin{eqnarray}\Vol\cQ_1(1^2, -1^2)=\cfrac{\pi^4}{3}=30\zeta(4) & \Vol\cQ_1(1^4)=\cfrac{\pi^6}{15}=63\zeta(6) \label{eq:ex1}\\
\Vol\cQ_1(2, -1^2)=\cfrac{4\pi^2}{3}=8\zeta(2) & \Vol\cQ_1(2, 1^2)=\cfrac{2\pi^4}{15}=12\zeta(4) \label{eq:ex2}\\
& \Vol \cQ_1(2, 2)=\cfrac{4\pi^2}{3}=8\zeta(2) \label{eq:ex3}
\end{eqnarray}
For an alternative computation of the volume of $\cQ(2, 1^2)$ using graphs, see Appendix \ref{app:alt}.
\end{Example}

\begin{rem}
In Section \ref{sect:EO} and Appendix \ref{tab}, the surfaces will be counted modulo symmetries. In particular it changes the volume of the third type of hyperelliptic components by a factor $1/2$ (hyperelliptic involution). For the two first types, labelling the zeros kills this symmetry, so the two conventions for the evaluation of the volumes coincide. 
\end{rem}

\begin{proof}

We recall here the three types of strata that contain hyperelliptic components (cf \cite{L1}): \begin{itemize}
\item First type: 
\[\xymatrix{\cQ^{hyp}(k_1^2 , k_2^2)
\ar[r]^\pi &
\cQ(k_1, k_2, -1^{2g+2})}\]

for $k_1\geq -1$ odd, $k_2\geq -1$ odd, $(k_1, k_2)\neq (-1, -1)$, ${g=\frac{1}{2}(k_1+k_2)+1}$ ($g$ is the genus of the surfaces in $\cQ(k_1, k_2, -1^{2g+2})$) . The map $\pi$ is a ramified double covering having ramifications points over $2g+2$ poles. Note that for $k_i=-1$ there are $2g+3$ poles and  ${2g+3}\choose{1}$ choices for the branch points in the base, so $2g+3$ choices for $\pi$. 
\item Second type: 
\[\xymatrix{\cQ^{hyp}(k_1^2 , 2k_2+2)
\ar[r]^\pi &
\cQ(k_1, k_2, -1^{2g+1})}\]

for $k_1\geq -1$ odd, $k_2\geq 0$ even, $g=\frac{1}{2}(k_1+k_2+3)$. The ramification points are $2g+1$ poles and the zero of order $k_2$. Note that for $k_1=-1$ there are $2g+2$ poles and  ${2g+2}\choose{1}$ choices for the cover $\pi$.
\item Third type:
\[\xymatrix{\cQ^{hyp}(2k_1+2 , 2k_2+2)
\ar[r]^\pi &
\cQ(k_1, k_2, -1^{2g})}\]

for $k_1$, $k_2$ even, $g=\frac{1}{2}(k_1+k_2)+2$. The ramification points are over all the singularities. \end{itemize}

We introduce the following notation common to the three types of hyperelliptic components: \[\cQ^{hyp}(\alpha)\underset{I:1}{\overset{\pi}{\longrightarrow}}\cQ(\beta)\] with $\alpha=(\alpha_1^{m_1}, \dots, \alpha_s^{m_s})$ and $\beta=( \beta_1^{n_1}, \dots,  \beta_r^{n_r})$, where $I$ is the number of choices for the cover: $I=1$ except for the special cases ($k_1=-1$) mentioned above.

Let $d=\dim_\C\cQ(\beta)$ be the complex dimension of the stratum that we consider.

Recall that, by definition, the volume of the hyperboloid of surfaces of area equal to $1/2$ is given by the volume of the cone underneath times the real dimension of the stratum:
\[\Vol\cQ_1(\beta)=2d\cdot\Vol\{S\in\cQ(\beta), \; \mathrm{area}(S)\leq 1/2\}\]

Let $S$ be a point in $\cQ_1(\beta)$, and let $S'$ be one of the $I$ possible lifts $\pi^*(S)$. As $S$ is of area $1/2$, $S'$ is of area $1$ so belongs to $\cQ_2^{hyp}(\alpha)$. So the cone underneath $\cQ_1(\beta)$ is in $1:I$ correspondence with the cone underneath $\cQ_2^{hyp}(\alpha)$. Now we want to compare the volume elements of $\cQ^{hyp}(\alpha)$ and $\cQ(\beta)$. So we have to understand how the lattice $(H^-_1(\hat S, \hat \Sigma;\Z))^{*}_{\C}$ is lifted by $\pi^*$ and compare it with the lattice $(H^-_1(\hat S', \hat \Sigma';\Z))^{*}_{\C}$, where $\hat S$ and $\hat S'$ are the orientation double covers of $S$ and $S'$ respectively. 

For the first type we have the following commutative diagram:
\[\xymatrix{\cH(k_1+1, k_2+1) 
\ar[d]  &  \ar@{.>}[l] \cH((k_1+1)^2 , (k_2+1)^2) \ar[d]\\
\cQ(k_1, k_2, -1^{2g+2})
 & \ar[l]_\pi^{I:1}
\cQ^{hyp}(k_1^2 , k_2^2)}\]

On $S\in \cQ(k_1, k_2, -1^{2g+2})$ we consider the saddle connections defined by taking a broken line joining all the singularities except one pole, as in the picture below, such that $a$ joins the two zeros, $b$ joins a zero to a pole, and $a_i,b_i$ join the remaining poles except the last one,  for $i$ going from $1$ to $g$. Then $\hat a, \hat b, \hat a_1, \dots, \hat b_g$ is a primitive basis of $H^-_1(\hat S, \hat \Sigma;\Z)$ (cf Lemma \ref{lem:basis}). On the other hand consider the saddle connections on $\cQ^{hyp}(k_1^2 , k_2^2)$ constructed using $a, b, a_1, \dots b_g$ in the following way: for all $a_i$ and $b_i$ and for $b$, take the combination of the two lifts by $\pi$ to obtain primitive cycles $A_i$, $B_i$, and $B$ in $H_1(S', \Sigma', \Z)$. Take only one of the two preimages of $a$ to get a primitive cycle $A$. Then $\hat A, \hat B, \hat A_1, \dots, \hat B_g$ define a primitive basis of $H^-_1(\hat S', \hat \Sigma';\Z)$ (same arguments as in Lemma \ref{lem:basis}).

\begin{figure}[h]
\includegraphics{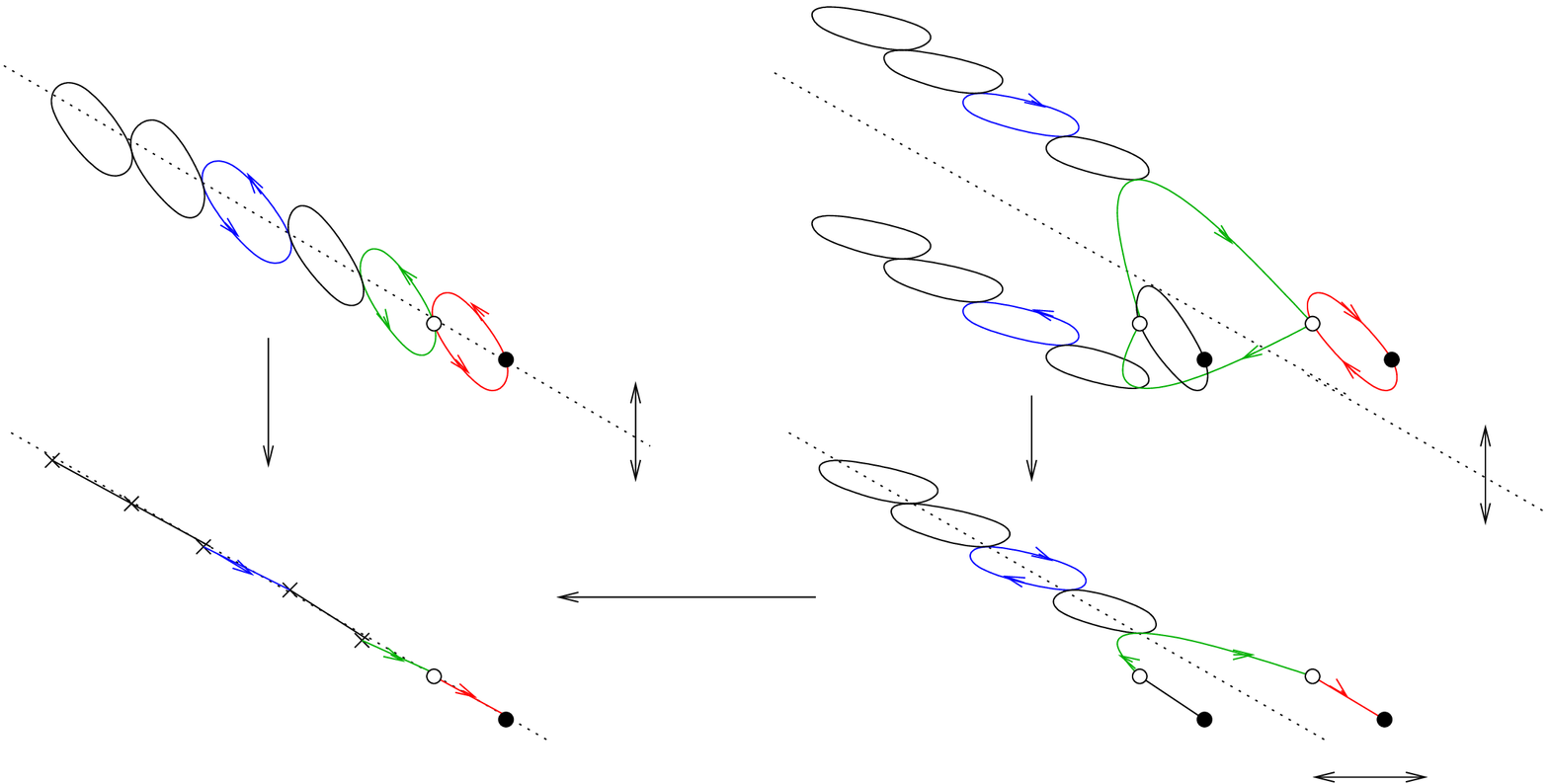}
\begin{picture}(200,130)
\put(0,0)
{\begin{picture}(0,0)
\put(-50,20){$\cQ(\beta)$}
\put(-50, 100){$\cH(\beta')$}
\put(200,20){$\cQ^{hyp}(\alpha)$}
\put(180, 100){$\cH(\alpha')$}
\put(70, 30){$\pi$}
\put(65, 60){$\sigma_d$}
\put(200, 55){$\sigma_u$}
\put(177, 0){$s$}
\put(-10, 36){\textcolor{blue}{$a_i$}}
\put(15, 22){\textcolor{green}{$b$}}
\put(30, 15){\textcolor{red}{$a$}}
\put(0, 110){\textcolor{blue}{$\hat a_i$}}
\put(20, 97){\textcolor{green}{$\hat b$}}
\put(35, 90){\textcolor{red}{$\hat a$}}
\put(135, 41){\textcolor{blue}{$A_i$}}
\put(160, 30){\textcolor{green}{$B$}}
\put(180, 23){\textcolor{red}{$A$}}
\put(130, 115){\textcolor{blue}{$\hat A_i$}}
\put(160, 100){\textcolor{green}{$\hat B$}}
\put(180, 83){\textcolor{red}{$\hat A$}}
\end{picture}}\end{picture}
\end{figure}

On the picture $\sigma_u$ and $\sigma_d$ are the involutions of the double covers and $s$ is the hyperelliptic involution. 

Since $\hat a$ is twice longer than $a$, the corresponding complex coordinate satisfies \[\dd\hat a=4\dd a.\]
So in local coordinates volume elements are given by:
\[\dd\nu_{down}=\dd\hat a\,\dd\hat b\,\dd\hat a_1\dots \dd\hat b_g=4^{d}\dd a\,\dd b\,\dd a_1\dots \dd b_g\]
and 
\[\dd\nu_{up}=\dd\hat A\,\dd\hat B\,\dd\hat A_1\dots \dd\hat B_g=4^{d}\dd A\,\dd B\,\dd A_1\dots \dd B_g\]
with $\dd A=\pi^*(\dd a)$, $\dd B=4\pi^*(\dd b)$, $\dd A_i=4\pi^*(\dd a_i)$ and $\dd B_i=4\pi^*(\dd b_i)$.

So we obtain the following relation between the volume elements:
\begin{equation}\label{eq:volel}\dd \nu_{up}=4^{d-1}\pi^*(\dd \nu_{down})
\end{equation}

The computation of $\dd\nu_{up}$ for the other types of connected components is completely similar to this case, and we get that Equation \eqref{eq:volel} holds in all cases.

So now we have all the elements to compute the relation between $\Vol\cQ_1(\beta)$ and $\Vol\cQ_1^{hyp}(\alpha)$:

\begin{eqnarray*}\Vol^{unnumb}\cQ_1^{hyp}(\alpha) & = &  2d\Vol^{unnumb}\{S'\in \cQ^{hyp}(\alpha), \; \mathrm{area}(S')\leq 1/2\}\\
 & = &  2d \cdot\frac{1}{2^d}\Vol^{unnumb}\{S\in \cQ^{hyp}(\alpha), \; \mathrm{area}(S)\leq 1\}\\
  & = &   \frac{2d}{2^d}\cdot I\cdot 4^{d-1}\Vol^{unnumb}\{S\in\cQ(\beta),\; \mathrm{area}(S)\leq 1/2\}\\
  & = & I\cdot 2^{d-2}\Vol^{unnumb}\cQ_1(\beta)\end{eqnarray*}

  Using Convention \ref{conv:label} and \eqref{eq:label} we get:
    \[\Vol\cQ_1^{hyp}(\alpha)= \cfrac{m_1!\dots m_s!}{|\Gamma^{hyp}(\alpha)|} \cdot I\cdot 2^{d-2}\cdot  \frac{\vert\Gamma(\beta)\vert}{n_1!\dots n_r!}\Vol\cQ_1(\beta)\]

Note that, for the first two types, the hyperelliptic involution (which is the only common symmetry to all surfaces in the component) exchanges the zeros which are preimages of the same zero downstairs. So for these types $\vert\Gamma^{hyp}(\alpha)\vert=2$. For the third type $\vert\Gamma^{hyp}(\alpha)\vert=1$, since the action of the hyperelleptic involution on the zeros is trivial.

Downstairs there is no symmetry for each stratum that we consider so $\vert\Gamma(\beta)\vert=1$ for each $\beta$.

For the special cases where $k_i=-1$, e.g. $\cQ^{hyp}(k_1^2, -1^2)\to \cQ(k_1, -1^{2g+3})$, the factor $I$ is exactly the multiplicity of the poles in the base ($2g+3$ in the example), so this factor is compensated by the factor $n_r!$ corresponding to the multiplicity of the poles in the numerator.

The values of the volumes of strata of quadratic differentials in genus 0 are given in \cite{AEZ}, Theorem 1.1:
\begin{equation}\Vol\cQ_1(\beta_1, \dots, \beta_n)=2\pi^2\prod_{i=1}^n v(\beta_i),\label{eq:volQgenus0}\end{equation}
with \[v(n)=\cfrac{n!!}{(n+1)!!}\cdot \pi^n\cdot \begin{cases} \pi & \mbox{when }n\mbox{ is odd}\\ 2& \mbox{when }n\mbox{ is even}\end{cases}\]
for $n\in\{-1, 0\}\cup\N$ and with $$n!!=n(n-2)(n-4)\cdots,$$ by convention $(-1)!!=0!!=1$.

In particular we have:
\begin{itemize}
\item for the first type ($k_1\geq -1$ odd, $k_2\geq -1$ odd, $(k_1, k_2)\neq (-1, -1)$, $d=2g+2$): \[\Vol\cQ_1(k_1, k_2, -1^{d})= 2\pi^{d}\frac{k_1!!}{(k_1+1)!!}\cdot\frac{k_2!!}{(k_2+1)!!},\] 
\item for the second type ($k_1\geq -1$ odd, $k_2\geq 0$ even, $d=2g+1$):
\[\Vol\cQ_1(k_1, k_2, -1^{d}) = 4\pi^{d-1}\frac{k_1!!}{(k_1+1)!!}\cdot\frac{k_2!!}{(k_2+1)!!},\] 
\item for the third type ($k_1$, $k_2$ even, $d=2g$):
\[\Vol\cQ_1(k_1, k_2, -1^{d}) = 8\pi^{d-2}\frac{k_1!!}{(k_1+1)!!}\cdot\frac{k_2!!}{(k_2+1)!!}.\] 
\end{itemize}
So we obtain the result.
\end{proof}

\subsection{Volumes of hyperelliptic components of strata of Abelian differentials}

Similarly we compute the volumes of the hyperelliptic components of Abelian differentials (for the needs of \cite{G}). To match with Convention \ref{convarea} we will consider $\cH_{1/2}$ the hypersurface of surfaces with area $1/2$. We follow also Convention \ref{conv:label} for these connected components and use the equality $\Vol\cH_{1/2}^{hyp}(\alpha)=\Vol^{numb}\cH_{1/2}^{hyp}(\alpha)$.

\begin{prop}The volumes of hyperelliptic components of strata of Abelian differentials with area $1/2$ are given by the following formulas:
\begin{eqnarray}\Vol\cH^{hyp}_{1/2}(k-1)=\cfrac{2^{k+2}}{(k+2)!}\cdot\cfrac{(k-2)!!}{(k-1)!!}\cdot\pi^{k+1}\label{eq:volHhyp1}\\
\Vol\cH^{hyp}_{1/2}\left(\left(\frac{k}{2}-1\right)^2\right)=\cfrac{2^{k+3}}{(k+2)!}\cdot\cfrac{(k-2)!!}{(k-1)!!}\cdot\pi^{k}\label{eq:volHhyp2}\end{eqnarray}
\end{prop}

\begin{rem}
Here again, if we choose to follow the Eskin-Okounkov convention and count surfaces modulo symmetries, the volume of the component ${\cH^{hyp}_{1/2}(k-1)}$ will be twice smaller.
\end{rem}

\begin{proof}
We recall here the two types of strata of Abelian differentials that contain hyperelliptic components (cf \cite{KZ}):
\begin{itemize}
\item First type ($ g\geq 2$): \[\xymatrix{\cH^{hyp}(2g-2)
\ar[r]^\pi &
\cQ(2g-3, -1^{2g+1})}\]

\item Second type ($g\geq 2$): \[\xymatrix{\cH^{hyp}((g-1)^2)
\ar[r]^\pi &
\cQ(2g-2, -1^{2g+2})}\]

 \end{itemize}
In both cases, $\pi$ is an isomorphism. By conventions \ref{convreseau} and \ref{convarea}, the volume elements are chosen to be invariant under this isomorphism, so we have:
\begin{eqnarray*}\Vol^{unnumb}\cH_1^{hyp}(2g-2)=\Vol^{unnumb}\cQ_1(2g-3, -1^{2g+1})\\
\Vol^{unnumb}\cH_1^{hyp}((g-1)^2)=\Vol^{unnumb}\cQ_1(2g-2, -1^{2g+2})\end{eqnarray*}

So considering the naming of the singularities we obtain:
\begin{eqnarray*}\Vol\cH^{hyp}_1(2g-2)& = &\cfrac{1}{(2g+1)!}\Vol\cQ_1(2g-3, -1^{2g+1})\\
& =& \cfrac{2}{(2g+1)!}\cdot\cfrac{(2g-3)!!}{(2g-2)!!}\cdot\pi^{2g}\\
\Vol\cH^{hyp}_1((g-1)^2)& =&\cfrac{2!}{2}\Vol^{unnumb}\cH^{hyp}_1((g-1)^2)\\
 & =&\cfrac{1}{(2g+2)!}\Vol\cQ_1(2g-2, -1^{2g+2})\\
  & =&\cfrac{4}{(2g+2)!}\cdot\cfrac{(2g-2)!!}{(2g-1)!!}\cdot\pi^{2g}\end{eqnarray*}

By plugging values of volumes given in (\ref{eq:volQgenus0}). 
For the first type, for $k=2g-1$ we have $\dim_\C\cH(k-1)=2g=k+1$. For the second type, for $k=2g$ we have $\dim_\C\cH\left(\left(\frac{k}{2}-1\right)^2\right) = 2g+1=k+1$.
Finally, note that \[\Vol\cH_{1/2}(\beta)=2^{\dim_\C\cH(\beta)}\Vol\cH_1(\beta).\]
\end{proof}


\section{Counting diagrams }\label{ssection:voldim5}
For strata of complex dimension $d\leq 5$, we follow the combinatorial approach introduced by Zorich (\cite{Z}) in the Abelian case, Athreya Eskin and Zorich (\cite{AEZ2}) in the quadratic case for genus 0.

The general idea is to count ``integer points'' in a large ball in the stratum, that is, surfaces corresponding to points of the normalization lattice in the stratum.

The relation between volume and number of lattice points is given in \S~2.3 of \cite{AEZ2}:
\begin{prop}[Athreya-Eskin-Zorich]
\begin{eqnarray}\Vol\cQ_1(\alpha) & = & 2d\cdot \lim\limits_{N\to\infty} N^{-d}\cdot\notag\\& &(\mbox{Number of lattice points of area at most }N/2\mbox{ in }\cQ(\alpha))\label{eq:volint}\end{eqnarray}
\end{prop}

Here we recall briefly the techniques of Athreya, Eskin and Zorich to count integer points (or square-tiles surfaces, or pillowcase covers) in genus 0, and explain how generalize them to higher genera.

A flat surface $(S,\omega)$ corresponding to an integer point, i.e. a point in the lattice $(H^-_1(\hat S, \hat \Sigma; \Z))^{*}_{\C}$ in local coordinates, can be decomposed into horizontal cylinders with half-integer or integer widths, with zeros and poles lying on the boundaries of these cylinders, that are called singular layers in \cite{AEZ2}. Each layer defines a ribbon graph (graph with a tubular neighborhood inside the surface), called {\it map} in combinatorics. A zero of order $\alpha_i$ belonging to a layer corresponds to a vertex of valency $\alpha_i+2$ in the associated graph, and edges of the graph emerging from this vertex correspond to horizontal rays emerging from the zero in the surface. The graph is metric: edges have half-integer lengths.
A ribbon graph or a map carries naturally a genus: it is the minimal genus of the surface in which it can be embedded. So a ribbon graph associated to a singular layer in $S$ has a genus lower or equal to the genus $g$ of $S$.
Also a ribbon graph has some faces corresponding to the connected components of its complementary in the minimal surface in which it can be embedded. In our case faces correspond to cylinders emerging from the layer. In genus $0$ each face corresponds to a distinct cylinder, in higher genus some cylinders may have both of their boundaries glued to the same layer.
For a ribbon graph $\Gamma$ we have the Euler relation:
\[ \chi_{\Gamma}=2-2g_{\Gamma}=V_{\Gamma}-E_{\Gamma}+F_{\Gamma}\]
where $g_{\Gamma}$ is the genus of $\Gamma$, $V_{\Gamma}$, $E_{\Gamma}$ and $F_{\Gamma}$ the number of vertices, edges and faces of $\Gamma$ respectively.
In the figure below we represent the two maps with one 4-valent vertex: one is of genus 0 and has 3 faces, the other is of genus 1 and has 1 face.
\begin{figure}[h]
\includegraphics{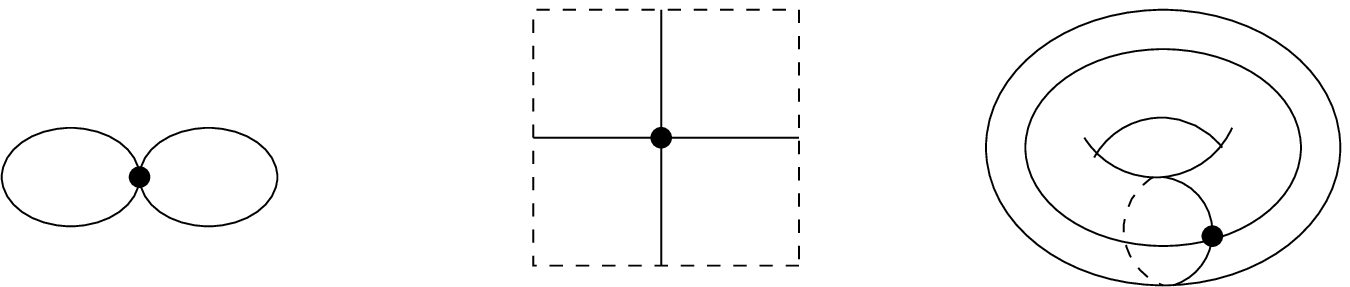}
\begin{picture}(200,60)
\put(0,0)
{\begin{picture}(0,0)
\put(-30,0){genus $0$}
\put(80,0){genus $1$}
\put(135,35){$=$}
\end{picture}}\end{picture}
\end{figure}

We encode the decomposition of the surface $S$ into horizontal cylinders in a supplementary graph $T$, by representing each singular layer by a point in this graph and each cylinder emerging from a layer by an edge emerging from the corresponding vertex. So a layer with $k$ faces corresponds to a $k$-valent vertex in $T$. We record also the information on the order of the zeros lying in each layer, and on the genus of the ribbon graph: that gives a decoration of the graph $T$. For surfaces $S$ of genus $0$ this graph is a tree.

As an example we consider a surface in $\cQ(2, 1^2)$ represented by the following graph:
\begin{figure}[h]
\includegraphics{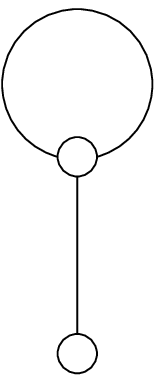}
\includegraphics{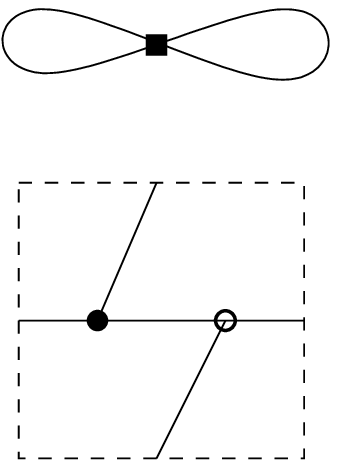}
\begin{picture}(200,80)
\put(0,0)
{\begin{picture}(0,0)
\put(14,12){\tiny{1}}
\put(-13, 11){$(1,1)$}
\put(20, 30){$w_1$}
\put(13, 75){$w_2$}
\put(14,52){\tiny{0}}
\put(-8, 51){$(2)$}
\put(123,16){$l_4$}
\put(114,35){$l_3$}
\put(143,16){$l_5$}
\put(110,58){$l_2$}
\put(140,58){$l_1$}
\end{picture}}\end{picture}
\end{figure}

On the left we figure the graph $T$. The lower vertex represents a ribbon graph of genus 1 with two zeros of order 1 (two 3-valent vertices): the corresponding layer is drawn on the right. The higher vertex corresponds to the ribbon graph of genus 0 with one 4-valent vertex (zero of order 2) drawn on the right. The width $w_i$ of the cylinders and the lengths $l_i$ of the edges of the ribbon graphs are also recorded.  

Below is a flat representation of a surface corresponding to the configuration described above.

\begin{figure}[h]
\includegraphics{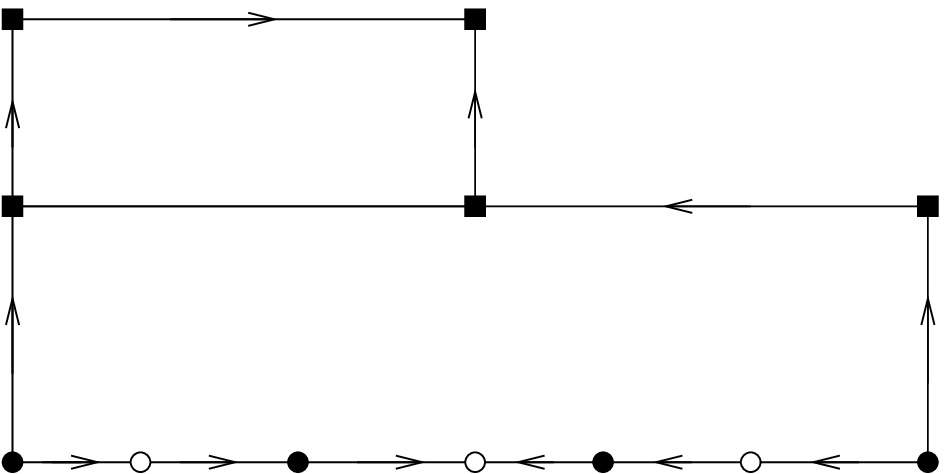}
\begin{picture}(200,80)
\put(0,0)
{\begin{picture}(0,0)
\put(35,80){1}
\put(35, 53){2}
\put(100, 53){1}
\put(10, 0){3}
\put(30,0){4}
\put(55,0){5}
\put(75, 0){3}
\put(95, 0){4}
\put(120, 0){5}
\end{picture}}\end{picture}
\end{figure}

Note that the genus of $S$ is the sum of the genera of the vertices of $T$, and the genus created by loops in the graph $T$: namely, the dimension of the homology of the graph $T$.
In the example, the surface is of genus 2.

Note also that horizontal cylinders in $S$ which are homologous to 0 correspond to separating edges on the graph $T$. It will be useful because with Convention \ref{convreseau}, the width $w$ of a cylinder is an integer if its waist curve is homologous to 0, and half-integer otherwise. In the example $w_1$ is integer and $w_2$ half-integer (furthermore here $w_1$ is necessarily equal to $2w_2$).

We have to choose the $l_i$ such that the length of the boundaries of the faces of the ribbon graphs $\Gamma_j$ correspond to the $w_k$. In the example we have necessarily $w_2=l_1=l_2$  and $w_1=2l_1=2w_1=2(l_3+l_4+l_5)$. So we have only one choice for $l_1$ and $l_2$ and exactly $\sum_{i=1}^{w_1-2}(i-1)=\frac{(w_1-2)^2}{2}$ choices for $(l_3, l_4, l_5)$ (see also Lemma \ref{lem:tool1} in Appendix \ref{app:toolbox}), because with the Convention $\ref{convreseau}$, $w_2$ is an integer and the $l_i$ are half-integer.

To count surfaces of area lower than $N/2$ corresponding to lattice points, we have to sum on the possible graphs $T$ and on the possible corresponding layers $\Gamma$, the number of distinct flat surfaces of this combinatorial type. So for a fixed graph $T$ and fixed layers $\Gamma_i$ we have to count the number of twists $t_j$, widths $w_i$, heights $h_i$ and lengths of saddle connexions $l_i$ satisfying the combinatorial configuration, and such that the area $w\cdot h=\sum_i w_ih_i$ is lower or equal to $N/2$. More precisely by (\ref{eq:volint}) we have to get the asymptotic of this number as $N$ goes to infinity. 
In the example all the $l_i$ are half-integer, $h_1, t_1, h_2, t_2$ also because they are coordinates of saddle connexions that are non homologous to zero, $w_2$ is half-integer and $w_1$ is integer. Twists $t_1$ and $t_2$ take respectively $2w_1$ and $2w_2$ half-integer values. We have already seen that the $l_i$ take $\frac{(w_1-2)^2}{2}$ values (with the condition $w_1=2w_2$). So we want to find the asymptotic when $N\to +\infty$ of \[\sum_{\substack{w_1h_1+w_2h_2\leq N/2\\w_1\in \N, \\w_2, h_1, h_2 \in\N/2}}2w_12w_2\frac{(w_1-2)^2}{2}\mathds 1_{\{w_1=2w_2\}}=\sum_{\substack{w(h_1+2h_2)\leq N/2\\w\in \N, h_1,\\ h_2\in \N/2}}8w^2\frac{(2w-2)^2}{2}\]
Remark that since we want only the term of highest order in $N$ we just need to take the term of highest order in $w_i$, so we can replace $\frac{(2w-2)^2}{2}$ by $\frac{(2w)^2}{2}=2w^2$. In general the asymptotic for such sums is given by Lemma 3.7 of \cite{AEZ2}. For this particular case, it is given by Lemma \ref{lem:tool2} of Appendix \ref{app:toolbox}, and we obtain $ \frac{N^5}{10}(32\zeta(4)-33\zeta(5))$.

This approach is somehow limited because we need to known all the ribbon graphs of a certain type and the number of these ribbon graphs increases fast as the dimension of the stratum grows. So we apply this method to strata of complex dimension $d\leq 5$, using the complete description of ribbon graphs with at most 5 edges given in \cite{JV}: recall that a zero of order $\alpha_i$ corresponds in the ribbon graph to a vertex with $\alpha_i+2$ adjacent edges, so the maximal number of edges of a ribbon graph in a stratum $\cQ(\alpha_1, \dots , \alpha_n)$ is $$\cfrac{\sum_{i=1}^{n}(\alpha_i+2)}{2}=2g-2+n=\dim_{\C}\cQ(\alpha_1, \dots, \alpha_n).$$

In genus 0, Athreya, Eskin and Zorich were able to compute the volumes of all strata of type $\cQ(1^k, -1^{k+4})$ with this method because they used a formula which gives directly the number of ways the cylinders of widths $w_i$ can be glued at a vertex $j$ of a tree $T$. This formula was deduced from a formula of Kontsevich by a recurrence on the number of poles.
The formula of Kontsevich works also for higher genus, but for distinct widths $w_i$, and since cylinders can form some loops in the surface, it is not obvious to get a general formula for the higher genus case, even for the strata $\cQ(1^k, -1^l)$. 

\begin{conv}In the following we write the half-integers in lower case and the integers in capitals.\end{conv}
\subsection{First example: $\cQ(5,-1)$}\label{sect:ex}

We use here the method described above to compute by a simple hand calculation the volume of $\cQ(5, -1)$. In this case, there are only  two possible graphs $T$, and for each graph, only two possible layers. This gives four configurations (note that here we do not speak about configurations of \^homologous cylinders, but about configurations of horizontal cylinders for integer surfaces in the stratum). The computations of the asymptotics are detailed in the appendix \ref{app:toolbox}.

\begin{itemize}
\item Configuration 1:\\
\begin{figure}[h!]
\includegraphics{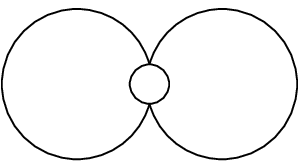}
\includegraphics{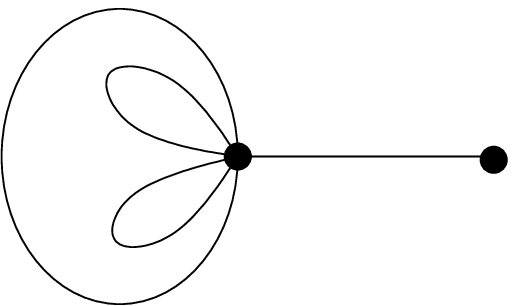}
\begin{picture}(200,60)
\put(0,0)
{\begin{picture}(0,0)
\put(29,33){\tiny{0}}
\put(13,44){$w_1$}
\put(42,44){$w_2$}
\put(95,12){$l_2$}
\put(115,12){$l_4$}
\put(115,32){$l_3$}
\put(160,20){$l_1$}
\end{picture}}\end{picture}
\caption{
\label{fig:config1}
Configuration 1
}
\end{figure} 

Convention \ref{convreseau} implies that all parameters $w_i, h_i, t_i, l_i$ are half-integers. 
The possible lengths of the waist curves of the cylinders are $l_3$,  $l_4$,  $l_2+2l_1$ and  $l_2+l_3+l_4$.
Since $l_2+l_3+l_4 >l_3$ and $l_2+l_3+l_4 >l_4$ we should have $l_3=l_4$ and $l_2+2l_1=l_2+2l_3$: $$\begin{cases}w_1=l_3=l_4\\ w_2=l_2+2l_1=l_2+2l_3\end{cases}  $$ 
There is one way to find such $(l_1,l_2,l_3,l_4)$, if $2w_1< w_2$.
The contribution to the counting for this configuration is:
$$\sum_{(w_1h_1+w_2h_2)\leq N/2} 4w_1w_2(\mathds{1}_{\{2w_1<w_2\}})=\sum_{(W_1H_1+W_2H_2)\leq 2N} W_1W_2(\mathds{1}_{\{2W_1<W_2\}})$$

\item Configuration 2:\\
\begin{figure}[h]
\includegraphics{plot/twoloops.eps}
\includegraphics{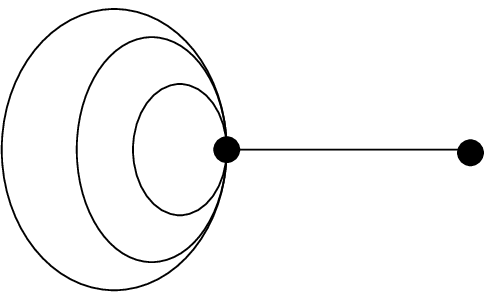}
\begin{picture}(200,60)
\put(0,0)
{\begin{picture}(0,0)
\put(29,33){\tiny{0}}
\put(13,44){$w_1$}
\put(42,44){$w_2$}
\put(95,12){$l_2$}
\put(108,22){$l_3$}
\put(121,32){$l_4$}
\put(160,20){$l_1$}
\end{picture}}\end{picture}
\caption{
\label{fig:config2}
Configuration 2
}
\end{figure} 

All parameters are half-integers.
The possible lengths for the waist curves of the cylinders are $l_4$,  $ l_3+l_4$ ,  $l_2+l_3$ and $ l_2+2l_1$. Since $l_3+l_4>l_4$ and the situation $$\begin{cases} l_4=l_2+2l_1\\ l_3+l_4=l_2+l_3\end{cases}$$ is impossible, the only remaining case is: $$\begin{cases} w_1=l_4=l_2+l_3\\ w_2=l_3+l_4=l_2+2l_1\end{cases}.$$
This implies that $ l_3=l_1 $ and there is only one way to find such $l_i$, but only if $w_1<w_2<2w_1$.
The contribution to the counting is:
$$\sum_{(w_1h_1+w_2h_2)\leq N/2} 4w_1w_2(\mathds{1}_{\{w_1<w_2<2w_1\}})=\sum_{(W_1H_1+W_2H_2)\leq 2N} W_1W_2(\mathds{1}_{\{W_1<W_2<2W_1\}})$$

Summing the contributions of the 2 first configurations gives: 
\begin{eqnarray*}\sum_{(W\cdot H)\leq 2N} W_1W_2(\mathds 1_{\{2W_1<W_2\}}+\mathds{1}_{\{W_1<W_2<2W_1\}}) & = & \sum_{W.H\leq 2N}W_1W_2\mathds 1_{\{W_1<W_2\}}\\ & \sim & \frac{1}{2}\frac{(2N)^4}{4!}(\zeta(2))^2=\cfrac{N^4(\zeta(2))^2}{3}\end{eqnarray*}

\item Configuration 3:\\
\begin{figure}[h]
\includegraphics{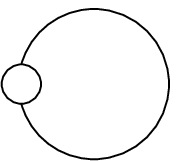}
\includegraphics{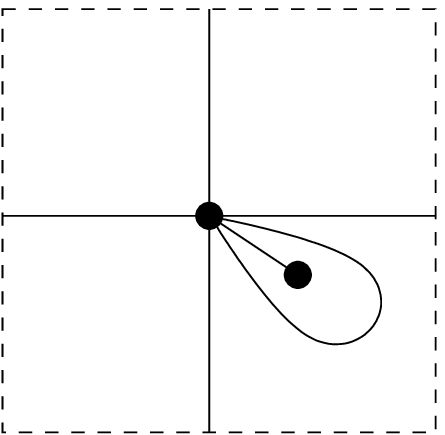}
\begin{picture}(200,70)
\put(0,0)
{\begin{picture}(0,0)
\put(3,33){\tiny{1}}
\put(14,44){$w$}
\put(162,6){$l_3$}
\put(115,30){$l_2$}
\put(128,60){$l_1$}
\put(152,18){$l_4$}
\end{picture}}\end{picture}
\caption{
\label{fig:config3}
Configuration 3
}
\end{figure} 

All parameters are half-integers.
The two lengths are $2l_1+2l_2+l_3$ and $l_3+2l_4$ so we should have $l_4=l_2+l_1$ in order that the two are equal. Then we search the number of $(l_1,l_2,l_3)$ such that $w=l_3+2(l_1+l_2)$. It is a polynomial of $w$ with leading term $\cfrac{1}{4}\cfrac{(2w)^2}{2}=\cfrac{w^2}{2}$.

The contribution to the counting is: $$\sum_{wh\leq N/2} 2\frac{w^3}{2}=\sum_{WH\leq 2N} \left(\frac{W}{2}\right)^3\sim \frac{1}{8}\frac{(2N)^4}{4}\zeta(4)=\cfrac{\zeta(4)}{2}N^4$$

\item Configuration 4:
\begin{figure}[h]
\includegraphics{plot/oneloop.eps}
\includegraphics{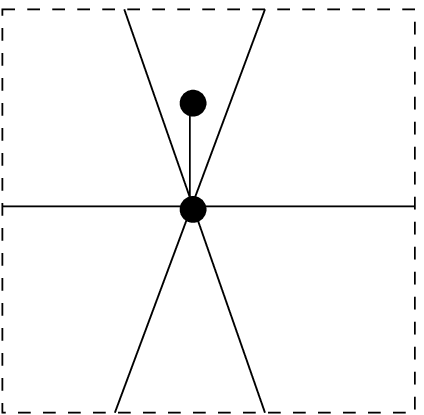}
\begin{picture}(200,70)
\put(0,0)
{\begin{picture}(0,0)
\put(3,33){\tiny{1}}
\put(14,44){$w$}
\put(120,16){$l_3$}
\put(115,30){$l_1$}
\put(130,60){$l_4$}
\put(144,16){$l_2$}
\end{picture}}\end{picture}
\caption{
\label{fig:config4}
Configuration 4
}
\end{figure} 

All parameters are half integers.
The lengths for the waist curves are $2l_1+l_2+l_3$ and $2l_4+l_2+l_3$, so we have $l_1=l_4$. The number of solutions of $w=2l_1+l_2+l_3$ is approximately $\cfrac{1}{2}\cfrac{(2w)^2}{2}=w^2$.

The contribution to the counting for this configuration: $$\sum_{wh\leq N/2}2 w^3=\sum_{WH\leq 2N}2 \left(\frac{W}{2}\right)^3=2\frac{1}{8}\frac{(2N)^4}{4}\zeta(4)=\zeta(4)N^4.$$
\item Total:\\
The sum of the 4 contributions is: $$N^4\left( \cfrac{(\zeta(2))^2}{3}+\frac{3}{2}\zeta(4)\right)=\cfrac{7\pi^4N^4}{2\cdot 3^3\cdot 5}$$
We obtain: $$\Vol Q(5,-1)=2\dim_{\C} Q(5,-1) \frac{7}{2\cdot 3^3\cdot 5}\pi^4=\frac{2^2\cdot 7}{3^3 \cdot 5}\pi^4$$

\end{itemize}

\subsection{Second example: $\cQ(3,-1^3)$}

As previously we compute the volume of this stratum by using the method described in \S~\ref{ssection:voldim5}.

\begin{itemize}
\item Configuration 1\\

\begin{figure}[h!]
\includegraphics{plot/oneloop.eps}
\includegraphics{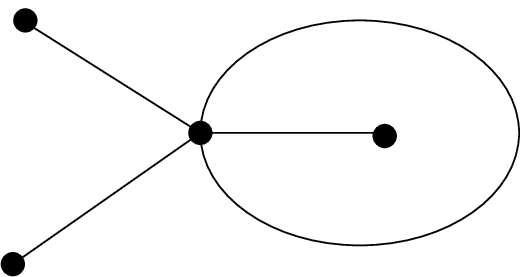}
\begin{picture}(200,60)
\put(0,0)
{\begin{picture}(0,0)
\put(2,34){\tiny{0}}
\put(13,52){$w$}
\put(115,52){$l_1$}
\put(150,15){$l_3$}
\put(115,22){$l_2$}
\put(160,40){$l_4$}
\end{picture}}\end{picture}
\caption{
\label{fig:3-1-1-1config1}
Configuration 1
}
\end{figure}

All parameters are half-integers.
The constraints are given by: $w=l_3+2l_4=l_3+2l_1+2l_2$. There are $\sim \frac{1}{4}\frac{(2w)^2}{2}=\frac{w^2}{2}$ choices for the $l_i$. There are 6 ways to give name to the poles.
The contribution to the counting is $$6\sum_{w.h\leq N/2}2w\frac{w^2}{2}=6\sum_{WH\leq 2N}\left(\frac{W}{2}\right)^3\sim\frac{3}{4}\frac{(2N)^4}{4} \zeta(4)=3\zeta(4)N^4$$

\item Configuration 2

\begin{figure}[h!]
\includegraphics{plot/ballon.eps}
\includegraphics{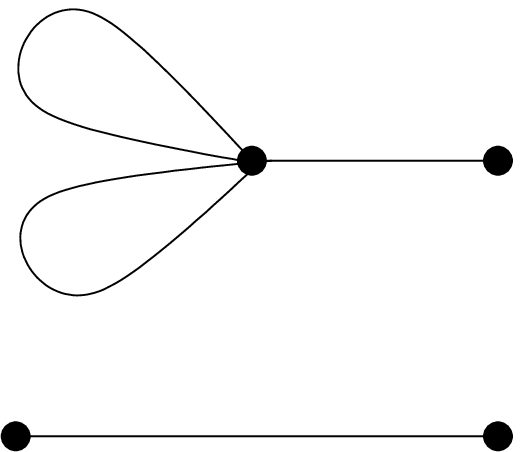}
\begin{picture}(200,60)
\put(0,0)
{\begin{picture}(0,0)
\put(14,32){\tiny{0}}
\put(14,-8){\tiny{0}}
\put(18,14){$w_1$}
\put(32,44){$w_2$}
\put(95,12){$l_3$}
\put(135,5){$l_4$}
\put(95,62){$l_2$}
\put(160,30){$l_1$}
\end{picture}}\end{picture}
\caption{
\label{fig:3-1-1-1config2}
Configuration 2
}
\end{figure}
The parameter $w_1=W_1$ is an integer and all remaining parameters are half-integers.
Note that here there are 3 ways to give names to the poles. The equations 
$$ \begin{cases}w_2=l_2=l_3\\ W_1=2l_1+l_2+l_3=2l_4\end{cases} $$
have one solution if $ W_1>2w_2 $.

The contribution of this configuration is:
$$ 3\sum_{W_1h_1+w_2h_2\leq N/2}2w_12w_2\mathds 1_{\{W_1>2w_2\}}=6\sum_{2W_1H_1+W_2H_2\leq 2N}W_1W_2\mathds 1_{\{W_1>W_2\}} $$

\item Configuration 3

\begin{figure}[h!]
\includegraphics{plot/ballon.eps}
\includegraphics{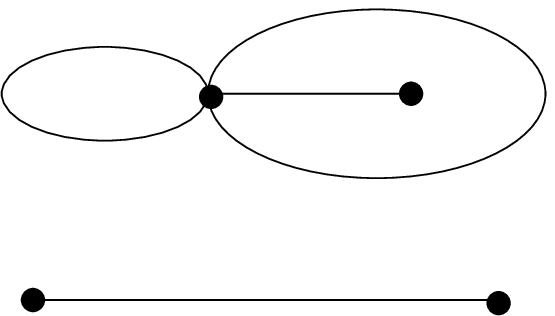}
\begin{picture}(200,60)
\put(0,0)
{\begin{picture}(0,0)
\put(14,32){\tiny{0}}
\put(14,-8){\tiny{0}}
\put(18,14){$w_1$}
\put(32,44){$w_2$}
\put(115,15){$l_1$}
\put(135,5){$l_4$}
\put(165,32){$l_3$}
\put(155,10){$l_2$}
\end{picture}}\end{picture}
\caption{
\label{fig:3-1-1-1config3}
Configuration 3
}
\end{figure}

The parameter $w_1=W_1$ is an integer and all remaining parameters are half-integers.
Note that here there are 3 ways to give names to the poles.

Two ribbon graphs are possible for the second layer:

\begin{figure}[h!]
\includegraphics{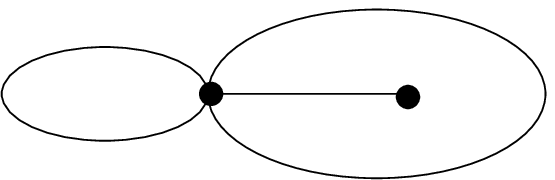}
\includegraphics{plot/3-1-1-1graph3ribbon.eps}
\begin{picture}(200,30)
\put(0,0)
{\begin{picture}(0,0)
\put(14,30){$1$}
\put(18,14){$2$}
\put(-30,10){$2$}
\put(115,10){$1$}
\put(155,14){$2$}
\put(160,30){$2$}
\end{picture}}\end{picture}

\end{figure}

For the first ribbon graph, the equations
$$\begin{cases} W_1=2l_4=l_1+l_2\\ w_2=l_1=l_2+2l_3\end{cases}$$
have one solution if $ w_2<W_1<2w_2 $.

For the second ribbon graph, the equations 
$$\begin{cases} W_1=2l_4=l_1\\ w_2=l_2+2l_3=l_2+l_1\end{cases}$$
have one solution if $ W_1<w_2 $.

The total number of solutions is then: $$ \mathds 1_{\{w_2<W_1<2w_2\}}+\mathds 1_{\{W_1<w_2\}}=\mathds 1_{\{W_1<2w_2\}}-\underset{negligible}{\underbrace{\mathds 1_{\{W_1=w_2\}}}} $$

This gives a contribution:
$$ 3\sum_{W_1h_1+w_2h_2\leq N/2}2w_12w_2\mathds 1_{\{W_1<2w_2\}}=6\sum_{2W_1H_1+W_2H_2\leq 2N}W_1W_2\mathds 1_{\{W_1<W_2\}}  $$

Summing the contributions of configurations 2 and 3 we get:
$$ 6\sum_{2W_1H_1+W_2H_2\leq 2N}W_1W_2=\frac{1}{4}6\sum_{W\cdot H\leq 2N}W_1W_2 \sim \frac{3}{2}\frac{(2N)^4}{4!}(\zeta(2))^2=\frac{5N^4}{2}\zeta(4)$$

\item Configuration 4

\begin{figure}[h!]
\includegraphics{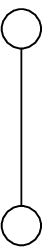}
\includegraphics{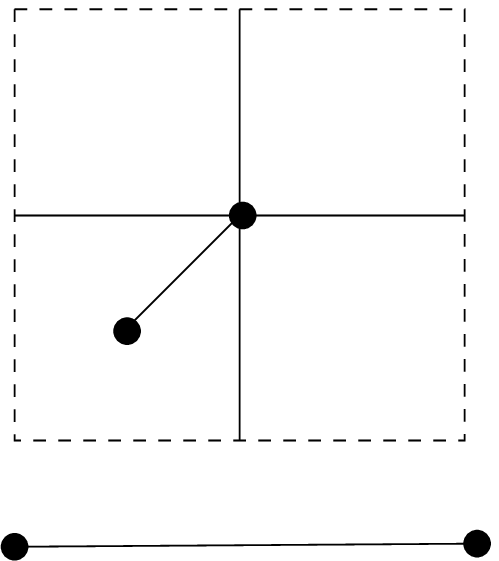}
\begin{picture}(200,60)
\put(0,0)
{\begin{picture}(0,0)
\put(3,53){\tiny{1}}
\put(3,12){\tiny{0}}
\put(6,34){$w$}
\put(135,-5){$l_4$}
\put(115,12){$l_3$}
\put(140,62){$l_1$}
\put(155,30){$l_2$}
\end{picture}}\end{picture}
\caption{
\label{fig:3-1-1-1config4}
Configuration 4
}
\end{figure}

The parameter $w=W$ is an integer and all remaining parameters are half-integers.
Note that here also there are 3 ways to give name to the poles.

The constraints are: $$ W=2l_4=2(l_1+l_2+l_3) $$ 
So there are $ \sim\frac{W^2}{2} $ ways to choose $ (l_1,\dots ,l_4) $.

The contribution of this configuration is:
$$ 3\sum_{W.h\leq N/2}2W\frac{W^2}{2}=3\sum_{WH\leq N}W^3\sim\frac{3N^4}{4}\zeta(4) $$

\item The sum of all contributions is $\frac{25N^4}{4}\zeta(4)$ so it gives 
\[\Vol^{comp}\cQ(3, -1^3)=50\zeta(4)=\frac{5\pi^4}{9}\]
\end{itemize}

\section{Computing generating functions following \cite{EO2}}\label{sect:EO}
In the Abelian case, volumes of strata were computed up to genus 20 by Eskin--Okounkov using representation theory and modular forms. In the quadratic case, they developed a similar theory but some additional difficulties arise for the computation of volumes.
The aim of this section is to recall the procedure to compute volumes using these results, to explain where the difficulties occur in the computations, to compute finally as many volumes as possible in the quadratic case, and to give the normalization factor between their convention and the \cite{AEZ}-convention.

In this section we introduce a new notation for simplicity:
\begin{defi} Let $\cQt(\alpha)$ denote the moduli space of pairs $(S,q)$ of Riemann surfaces $S$ and meromorphic quadratic differentials $q$ with exactly $n$ singularities of orders given by $\alpha_1, \dots \alpha_n$, where $q$ is allowed to be a global square of an Abelian differential.
\end{defi}
Note that if there is at least one zero of odd multiplicity then $\cQt(\alpha)=\cQ(\alpha)$ otherwise $\cQt(\alpha)=\cQ(\alpha)\cup\cH(\alpha/2)$.

\subsection{Convention for the normalization of the volume: description of the lattice}\label{convEO}
The convention of Eskin and Okounkov for the normalization of the volume element is slightly different from the previous one, due to Athreya, Eskin and Zorich. In particular the ``integer points'' in the strata will be also tiled by squares, but the chosen lattice differs. In fact here lattice points are covers of the torus in the Abelian case, and covers of the pillow in the quadratic case, with some constraints that we recall here.

\subsubsection{Abelian case}
Let $ \T^2 =\C/(\Z+i\Z)$ be the standard torus.
For a given stratum $ \cH(\beta)=\cH(\beta_1, \dots, \beta_n) $, fix $ n $ points $P_i$ in $ \T^2 $, and denote $ \mu_i=\beta_i+1 $ for $ i=1\dots n $. Then the chosen lattice for this stratum is the following: 
\begin{align*}L_{ab}(\cH(\beta))=\{S\in\cH(\beta); \; S\mbox{ is a cover of }\T^2\mbox{ ramified over each }P_i\\
\mbox{ with ramification profile }\mu_i \}\end{align*}
We denote \[\Cov_d(\mu)=\Card\{S\in L_{ab}(\cH(\beta)), S\mbox{ is of degree }d\}.\]
We introduce also the number \[w(\mu)=|\mu|+l(\mu),\] where $l(\mu)$ is the number of parts in $\mu$.
\subsubsection{Quadratic case}
Let $\cQt(\alpha)$ be a stratum of quadratic differentials. The set of singularities $(\alpha_1, \dots, \alpha_n)$ corresponds a couple of partitions $(\mu,\nu)$ by the following formulas: assume that the even zeros are the $b$ first ones, then we define \begin{align*}\mu_i =\frac{\alpha_{i}}{2}+1\mbox{ for }i=1\dots b\\ \nu_i  = \alpha_{i+b}+2\mbox{ for }i=1\dots n-b.\end{align*}
This gives a 1:1 correspondence between sets of singularity orders of quadratic differentials and couples of partitions, the second being a partition of an even number into odd parts (correspondence between \cite{AEZ} and \cite{EO2} notation).
Let $\fB=\T^2/\pm$ (called ``pillow'') and fix $b$ points $P_i$ on it (outside of the corners). In this setting, the chosen lattice is the following: 

\begin{align*}L_{quad}(\cQt(\alpha))=\{S\in\cQt(\alpha); \;\exists d>0, \;  S\mbox{ is a }2d\mbox{ cover of }\fB\mbox{ ramified }\\
\mbox{ over each }P_i\mbox{ with ramification profile }(\mu_i,1^{2d-\mu_i}), \\ \mbox{over }0\mbox{ with ramification profile }(\nu, 2^{d-|\nu|/2})\\ \mbox{ and over the three other corners with ramification profile }(2^d) \}\end{align*}

We denote \[\Cov_{2d}(\mu,\nu)=\Card\{S\in L_{quad}(\cQt(\alpha)), S\mbox{ is of degree }2d\}.\]
We introduce also the following number:
$$w(\mu, \nu)=|\mu| +l(\mu)+|\nu|/2.$$ 
We can express all data for $S\in L_{quad}(\cQt(\alpha))$ in terms of $\mu$ and $\nu$:
\begin{itemize}
\item genus $g=\frac{1}{2}(|\mu| -l(\mu) +|\nu|/2 -l(\nu))+1\;(*)$
\item genus of the double cover $\hat g=|\mu| -l(\mu) +|\nu|/2-l(\nu)/2+1$
\item efficient genus $g_{\eff}=\frac{1}{2}(|\mu|-l(\mu)+|\nu|/2)$
\item complex dimension $\dim_\C=|\mu|+|\nu|/2$
\end{itemize}

\subsection{Computation of volumes in the Abelian case}
We recall here some of the results of \cite{EO} that are used to compute volumes.
Let introduce the following generating functions (here we modify the notations of \cite{EO} into notations of \cite{EO2}):
\begin{align*}
Z(\mu; q)=\sum_{d\geq 1} \Cov_d(\mu)q^d\\
Z'(\mu; q)=\sum_{d\geq 1} \Cov'_d(\mu)q^d=\cfrac{Z(\mu;q)}{Z(\emptyset; q)}
\end{align*}
that enumerate covers and covers without unramified components respectively. Here \[Z(\emptyset; q)=\prod_{n\geq 1} (1-q^n)^{-1}\] is the generating function for the unramified  coverings.

Finally we denote $$Z^{\circ}(\mu; q)=\sum_{d\geq 1} \Cov^{\circ}_d(\mu)q^d$$  the generating function for the connected coverings.

Introducing the $q$-bracket of a shifted symmetric function $F$: \[\langle F\rangle_q=\cfrac{1}{Z(\emptyset; q)}\sum_{\lambda\in\Pi}q^{|\lambda|}F(\lambda)\] where $\Pi$ denote the set of partitions, Eskin and Okounkov showed (Proposition 2.11 in \cite{EO}):
\begin{prop}
\[Z'(\mu; q)=\langle f_{\mu_1}\dots f_{\mu_n}\rangle_q,\]
where $f_{\mu_i}(\lambda)=f_{\mu_i, 1, \dots , 1}(\lambda)$ is the central character of an element of cycle-type $(\mu_i, 1, \dots, 1)$ in the representation $\lambda$.
\end{prop} 
The algebra $\Lambda^*$ of shifted symmetric functions is generated by the functions:
\[p_k(\lambda)=\sum_{i=0}^\infty [(\lambda_i-i+\frac{1}{2})^k-(i+\frac{1}{2})^k]+(1-2^{-k})\zeta(-k).\]
The decomposition of the functions $f_{\mu_i}$ in term of the $p_k$ is known (see \cite{Ls} for example), so the $q$-brackets of products of function $f_i$ are polynomials in the $q$-brackets of products of functions $p_k$, that are quasi-modular forms of weight $w(\mu)$ (see \cite{EO} \S 5.1).

The generating function $Z'$ is then totally described, and so is $Z^{\circ}$ by inclusion-exclusion (cf Proposition 2.11 of \cite{EO}). To extract from this generating function the values of the volumes they show that (Proposition 1.6 and Proposition 3.2):
\begin{prop}
\[Z^\circ(\mu;q)\sim \frac{\dim_\R(\cH(\beta))\Vol(\cH_1(\beta))\cdot|\mu|!}{(1-q)^{|\mu|}}\;\mbox{ as }q\to 1\]
\end{prop}

Their method to compute the volumes is then the following:
\begin{itemize}
\item They compute the coefficient corresponding to the highest weight in the decomposition of the $f_{\mu_i}$ in the algebra basis of $p_k$ (Theorem 5.5)
\item They compute the highest term in the asymptotic of the $q$-brackets of products of $p_k$ as $q$ goes to 1 (Theorem 6.7)
\item They obtain the volume thanks to the previous proposition (Proposition 1.6 and 3.2.)
\end{itemize}

\subsection{Computation of volumes in the quadratic case}
First let us recall the main results of \cite{EO2}, and then let us detail the computations in this case.
Similarly to the case of Abelian differentials, we introduce the following generating functions:
\begin{align*}
Z(\mu,\nu; q)=\sum_{d\geq 1} \Cov_{2d}(\mu,\nu)q^{2d}\\
Z'(\mu,\nu; q)=\sum_{d\geq 1} \Cov'_{2d}(\mu,\nu)q^{2d}=\cfrac{Z(\mu,\nu;q)}{Z(\emptyset,\emptyset; q)}\\
Z(\emptyset,\emptyset; q)=\prod_{n\geq 1} (1-q^{2n})^{-1/2}\\
Z^\circ(\mu, \nu;q)=\sum_{d}Cov^\circ_{2d}(\mu, \nu)q^{2d}
\end{align*}
enumerating the covers, the covers without unramified connected components,the unramified covers, the connected covers respectively. 

The algebra $\Lambda^*$ of shifted symmetric functions is now enlarged to the algebra $\overline\Lambda$ generated by the functions $p_k$ as before and the functions $\overline{p_k}$ defined by:

\[\overline{p_k}(\lambda)=\sum_{i=1}^\infty \left[(-1)^{\lambda_i-i+1}(\lambda_i-i+\frac{1}{2})^k-(-1)^{-i+1}(-i+\frac{1}{2})^k\right]+c_k,\]
where the $c_k$ are determined by the expansion
\[\sum_k \cfrac{z^k}{k!}\overline{p_k}(\emptyset)=\cfrac{1}{e^{z/2}+e^{-z/2}}.\]

For any function $F$ the authors of \cite{EO2} introduce the $\w$-bracket as:
\[\langle F\rangle_\w=\cfrac{1}{Z(\emptyset, \emptyset; q)}\sum_{\lambda\in B\Pi}q^{|\lambda|}\w(\lambda)F(\lambda),\]
with \[\w(\lambda)=\left(\cfrac{\dim \lambda}{|\lambda|!}\right)f_{2,2,\dots, 2}(\lambda)\] for $\lambda\in B\Pi$, where $B\Pi$ denote the set of balanced partitions, that is, partitions $\lambda$ such that $\overline{p}_0(\lambda)=1/2$. For the aim of this section we only need to resume the results of \cite{EO2}, so we do not explain the possible interpretation of the objects that we consider.

The authors of \cite{EO2} show the following formula, similar to the Abelian case:
\begin{prop}\label{prop:Z'}
\[Z'(\mu, \nu; q)=\left\langle \cfrac{f_{\nu, 2, 2, \dots, 2}}{f_{2, 2, \dots, 2}}\prod_i f_{\mu_i}\right\rangle_\w\]
\end{prop}
The underlying sum in this formula begins for partitions of $\max(|\nu|, \mu_i)$.

Similarly to the Abelian case we can extract the volumes from the asymptotic of the generating function as $q\to 1$. Assume first that $\cQt(\alpha)=\cQ(\alpha)$. 
\begin{prop} Let $\dim_\C=\dim_\C(\cQ(\alpha))$ and $\dim_\R=2\dim_\C$. Then:
\[Z^\circ(\mu, \nu;q)\sim \frac{\Vol^{EO}(\cQ_1(\alpha))}{\dim_\R}\frac{(\dim_\C)!}{(1-q)^{\dim_\C}}\;\mbox{ as }q\to 1\]
\end{prop}
\begin{proof}
Introducing 
\[Z^\circ_{2D}(\mu, \nu)=\sum_{d=1}^{D} Cov^\circ_{2d}(\mu, \nu)\]
(recall that 
and following the proof of \cite{EO}, Prop 1.6 we get:
\[Z^\circ_{2D}(\mu, \nu)\sim\rho(\cQ_1(\mu, \nu))(2D)^{\dim_\C}\;\mbox{ as }D\to\infty,\]
where \[\rho(\cQ_1(\mu, \nu))=\cfrac{\Vol^{EO}(C(\cQ_1(\alpha)))}{\dim_\R(\cQ(\alpha))}.\]

Following the proof of Prop 3.2 in \cite{EO} we get:
\begin{align*}
\cfrac{1}{1-q^2}Z^\circ(\mu, \nu;q)&=\sum_{d=1}^\infty q^{2d}Z^\circ_{2d}(\mu, \nu)\\
& \sim \sum_{d=1}^\infty q^{2d}\rho(\cQ_1(\mu, \nu))(2d)^{\dim_\C}\;\mbox{ as }q\to 1\\
&\sim \rho(\cQ_1(\mu, \nu))\sum_{d}q^{2d}d^{\dim_\C}\;\mbox{ as }q\to 1\\
&\sim \rho(\cQ_1(\mu, \nu)) \cfrac{2^{\dim_\C}\Gamma(\dim_\C+1)}{(1-q^2)^{\dim_\C+1}}\;\mbox{ as }q\to 1
\end{align*}
which ends the proof.
\end{proof}
If $\cQ(\alpha)\neq\cQt(\alpha)$ then one should first extract the purely quadratic contribution $Z^\circ_{quad}(\mu, \nu; q)$ (see \S \ref{step4}), and then the same result holds for $Z^\circ_{quad}(\mu, \nu; q)$.

The method for the Abelian case does not applied here, because there is no equivalent of Theorem 5.5 and Theorem 6.7 of \cite{EO} here (see \cite{R-Z}).
Let us explain how to compute the volumes in this case.

The principal result of their article (Theorem 1 of \cite{EO2}) is:
\begin{theo}[Eskin--Okounkov]
$Z'(\mu, \nu;q)$ is a polynomial in $E_2(q^2)$, $E_2(q^4)$, and $E_4(q^4)$ of weight $w(\mu, \nu)$
\end{theo}

Examples of such generating functions are given in Appendix A of \cite{EO2}.

The procedure to compute volumes is then the following:
\begin{enumerate}
\item Compute the coefficients of the polynomial $Z'$ in $E_2(q^2)$, $E_2(q^4)$, $E_4(q^4)$ (see \S\ref{step1})
\item Deduce $Z^\circ$ from $Z'$ (see \S\ref{step2})
\item Compute the asymptotic development as $q$ goes to 1 of $Z^\circ$ (see \S\ref{step3})
\end{enumerate}

The first step constitutes the main part of the computations. We explain first how to make the last step, since it is the easiest.

Two additional steps are required to compare these volumes to the previous computed ones, they are described in \S \ref{step4} and \ref{step5}.

\subsection{Step 3: Computing the asymptotic development of $Z^\circ$}\label{step3}
After the second step (see \S\ref{step2}), we obtain $Z^\circ$ as a polynomial in $E_2(q^2)$, $E_2(q^4)$, $E_4(q^4)$. 

Let $q=e^{2i\pi\tau}$, $\tilde q=e^{_i\pi/2\tau}$, and $h=-2i\pi\tau$ so $q=e^{-h}$. We use the (quasi)-modular transformations:
\begin{eqnarray*}
E_2(q^2)=-\frac{\pi^2}{h^2}E_2({\tilde q}^2)-\frac{1}{4h}\\
E_2(q^4)=-\frac{\pi^2}{4h^2}E_2(\tilde q)-\frac{1}{8h}\\
E_4(q^4)=\frac{\pi^4}{16h^4}E_4(\tilde q)
\end{eqnarray*}
Finding the asymptotic development as $q\to 1$ is equivalent to finding the asymptotic development as $h\to 0$.

Recall that with the convention of \cite{EO2}, we have the following developments:
\begin{eqnarray*}
E_2(q)=-\frac{1}{24}+q+3q^2+4q^3+7q^2+6q^5+12 q^6 + \dots\\
E_4(q)=\frac{1}{240}+q+9q^2+28 q^3+73 q^4+126 q^5+252 q^6+\dots
\end{eqnarray*}

Note that, except for the constant terms, all terms in the development of $E_2(\tilde q^2)$, $E_2(\tilde q)$, and $E_4(\tilde q)$ are negligible compared to any power of $h$ as $h\to 0$.

It means that making the following replacements:
\begin{equation}\begin{cases}
E_2(q^2)\longleftrightarrow \cfrac{\pi^2}{24\cdot h^2}-\cfrac{1}{4h}\\
E_2(q^4)\longleftrightarrow\cfrac{\pi^2}{4\cdot 24\cdot h^2}-\cfrac{1}{8h}\\
E_4(q^4)\longleftrightarrow\cfrac{\pi^4}{16\cdot 240 \cdot h^4}
\end{cases}
\end{equation}\label{eq:mod}
we obtain exactly the asymptotic development of $Z^\circ$ as $h\to 0$.

\subsection{Step 2: From possibly disconnected covers to connected covers}\label{ssect:con}\label{step2} 
Define a {\it substratum} of a stratum $\cQ(\alpha)$ as a stratum which singularity orders belong to the set $\{\alpha_1, \dots \alpha_n\}$ of singularity orders of $\cQ(\alpha)$. Define a {\it decomposition of $\cQ(\alpha)$ into substrata} as an union of substrata of $\cQ(\alpha)$, such that the total set of singularity orders corresponds to $\{\alpha_1, \dots \alpha_n\}$. For example $\cQ(-1^4)\cup \cQ(4, 1^2, -1^2)$ is a decomposition of $\cQ(4, 1^2, -1^6)$ into substrata. For a fixed $\cQ(\alpha)$, the set of decompositions of $\cQ(\alpha)$ is naturally partially ordered.  
For example, here is the diagram of the poset of the decompositions of $\cQ(4, 1^2, -1^6)$ into substrata:


  \begin{displaymath}
    \xymatrix @!0 @R=0.4cm @C=2.5cm{
   & \cQ(4,1^2, -1^6) \ar@{-}[rddd] \ar@{-}[rrddd] \ar@{-}[lddd] \ar@{-}[dd]&& \\ &&&
   \\ &&&
   \\   
\cQ(4, -1^4) \,\cup
       &\cQ(-1^4) \,\cup &\tilde\cQ(4) \,\cup    &  \cQ(1, -1^5)\, \cup \\
 \cQ(1^2, -1^2) \ar@{-}[rddd] &  \cQ(4, 1^2, -1^2) \ar@{-}[ddd] &\cQ(1^2, -1^6)\ar@{-}[lddd] &\cQ(4,1,-1)
      \\ &&&
      \\ &&&
      \\
         & \tilde\cQ(4)\cup \cQ(1^2, -1^2)\cup \cQ(-1^4) & &
  }
\end{displaymath} 

Such a poset possess a well-defined M\"obius function $\mu$, defined as the inverse of the zeta function $\zeta(x,y)=1 \; \forall x\leq y$ (see Chapter 3 of \cite{St} for a reference on posets and M\"obius functions). From the decomposition formula
\[Z'(x)=\sum_{y\leq x}Z^\circ(y)\; \forall x\] we deduce in particular, using the M\"obius inversion: \[Z^\circ(\hat 1)=\sum_{y\leq \hat 1}Z'(y)\mu(y, \hat1),\] which is precisely the wanted formula. For the previous example we have the corresponding values of $\mu(y, \hat 1)$, for $y$ element of the poset:
  \begin{displaymath}
      \xymatrix{
   & 1 \ar@{-}[rd] \ar@{-}[rrd] \ar@{-}[ld] \ar@{-}[d]&& \\
-1      \ar@{-}[rd] &-1 \ar@{-}[d]&-1 \ar@{-}[ld] & -1 \\
         & 2 & &
  }\end{displaymath}

  For this example note that there is no difference with the Abelian case (\S 2.2 in \cite{EO}). But in general, since there are some symmetries in the decomposition into substrata, and since the even zeros are numbered, we need to modify the M\"obius function. We inverse  the more general formula
  \[Z'(x)=\sum_{y\leq x}Z^\circ(y)a(y,x)\; \forall x\]  where $a(x,x)\neq 0 \; \forall x$, using the inverse of the function $a$ that we denote $ \mu_a$ (which exists, Proposition 3.6.2 of \cite{St}). The function $a$ takes care of the possible symmetries and the numbering of the even zeros. As in the classical case the values of the M\"obius function $\mu_a(y, \hat 1)$ can be computed recursively using the relation \[(\mu_a a)(x,y)=\sum_{x\leq z\leq y}\mu_a(x,z)a(z,y)=\delta(x,y).\] 
  
  Instead of writing a complicated general formula for the function $a$ we prefer to explicit this function for three representative examples. 
  
\subsubsection{Symmetries: $\cQ(3^2, -1^{10})$ and $\cQ(8, -1^{12})$}
The symmetries in the decomposition into substrata occur when some substrata are equal. For example the stratum $\cQ(3^2, -1^6)$ decomposes into $\cQ(3, -1^3)\cup\cQ(3, -1^3)$, so the decomposition into connected components is:
\[Z'(\emptyset, [5^2, 1^6])=Z^\circ(\emptyset, [5^2, 1^6])+Z^\circ(\emptyset, [5^2, 1^2])Z^\circ(\emptyset, [1^4])+\frac{1}{2}Z^\circ(\emptyset,[5,1^3])^2\]
More generally we have to divide by the cardinality of the symmetric group that permutes the equal components in the counting, in order not to count the same surface twice.
For the stratum $\cQ(3^2, -1^{10})$ we have the following decomposition:

 \begin{displaymath}
      \xymatrix @!0 @R=0.4cm @C=2.5cm{
   & A=\cQ(3^2,-1^{10}) \ar@{-}[rddd]  \ar@{-}[lddd] && \\
   &&&\\
   &&&\\
    B= \cQ(3, -1^3)\cup\cQ(3, -1^7)
      \ar@{-}[rddd] &&C=\cQ(-1^4)\cup \cQ(3^2, -1^6) \ar@{-}[lddd] \ar@{-}[rddd]& \\
      &&&\\
      &&&\\
         & D=\cQ(3, -1^3)\cup\cQ(3, -1^3) & &E=\cQ(3^2, -1^2)\cup\\
     & \cup\,\cQ(-1^4)   &&\cQ(-1^4)\cup\cQ(-1^4)
  }\end{displaymath}
We give below the table for the function $a$ and the corresponding values of $\mu_a(y,\hat 1)$. 
  \begin{displaymath}\begin{array}{m{5cm}m{5cm}}
         $ \begin{array}{c|ccccc}
        a(y,x)& A&B&C&D&E\\
        \hline
        A&1&&&&\\
        B&1&1&&&\\
        C&1&  & 1 &&\\
        D& 1/2 & 1 & 1/2 & 1&\\
        E&1/2 & & 1 & & 1
        \end{array}$
&
        \xymatrix @R=0.5cm @C=0.5cm{
   & 1 \ar@{-}[rd]  \ar@{-}[ld] && \\
    -1  \ar@{-}[rd] &&-1 \ar@{-}[ld] \ar@{-}[rd]& \\
         &  1& &\cfrac{1}{2} 
  }\end{array}\end{displaymath}
This diagram gives the following inclusion-exclusion formula:
\begin{align*}Z^\circ(\emptyset, [5^2, 1^{10}])=&Z'(\emptyset, [5^2, 1^{10}])-Z'(\emptyset, [5,1^3])Z'(\emptyset, [5, 1^7])\\&-Z'(\emptyset, [1^4])Z'(\emptyset, [5^2, 1^6])+Z'(\emptyset, [5,1^3])^2Z'(\emptyset, [1^4])\\&+\frac{1}{2}Z(\emptyset,[5^2,1^2])Z'(\emptyset, [1^4])
\end{align*}
For the stratum $\cQ(8, -1^{12})$, the group $\mathfrak S_3$ acts on the three substrata $\cQ(-1^4)$, which gives a factor $1/6$.
  
        \begin{displaymath}
        \xymatrix  @R=0.5cm {
    A=\cQ(8, -1^{12}) \ar@{-}[d]  \\
        B=\cQ(8, -1^8)\cup\cQ(-1^4)\ar@{-}[d]  \\
        C=\cQ(8, -1^4)\cup\cQ(-1^4)\cup\cQ(-1^4)\ar@{-}[d]  \\
        D=\cQ(8, -1^4)\cup\cQ(-1^4)\cup\cQ(-1^4)\cup\cQ(-1^4)}
        \end{displaymath}
\begin{displaymath}\begin{array}{m{4.5cm}m{1.5cm}}
        $\begin{array}{c|cccc}
        a(y,x)& A&B&C&D\\
        \hline
        A&1&&&\\
        B&1&1&&\\
        C&1/2 & 1 & 1 &\\
        D& 1/6 & 1/2 & 1 & 1
        \end{array}$
&        
 \xymatrix @R=0.5cm{
    1 \ar@{-}[d]  \\
        -1\ar@{-}[d]  \\
        1/2\ar@{-}[d]  \\
        -1/6}
        \end{array}
  \end{displaymath}

\subsubsection{Numbering of the even zeros: $\cQ(2^3, -1^2)$}  
Noting that the even zeros are numbered by definition, because they arise as branching points over numbered distinct ramification points on the base, we see that there are three ways to decompose the stratum $\cQ(2^3, -1^2)$ into $\cQ(2^2)\cup\cQ(2, -1^2)$. So we obtain the following functions $a$ and $\mu_a$ for this stratum:
    \begin{displaymath}\begin{array}{m{4cm}m{3cm}m{3cm}}
        \xymatrix{
    A=\cQ(2^3, -1^2) \ar@{-}[d]  \\
        B=\cQ(2^2)\cup\cQ(2, -1^2)}
&
        $\begin{array}{c|cc}
        a(y,x)& A&B\\
        \hline
        A&1&\\
        B&3&1\\
        \end{array}$
&
 \begin{xy}
(20,0)*+{1}="A";
(12,-15)*+{-1}="B";
(20,-15)*+{-1}="C";
(28,-15)*+{-1}="D";
(20,-21)*+{-3};
{\ar@{-} "A";"B"};
{\ar@{-} "A";"C"};
{\ar@{-} "A";"D"};
(10,-7.5)*\ellipse(12,3.5){-};
\end{xy}\end{array}
  \end{displaymath}

\subsection{Step 1: Computing the generating function as a polynomial in quasi-modular forms}\label{step1}

\subsubsection{First method}
The first method consists to apply naively Proposition \ref{prop:Z'} and compute the first terms in the development. We denote $QMF_2(\Gamma_0(2))$ the algebra of quasi-modular forms generated by $E_2(q^2)$, $E_2(q^4)$, $E_4(q^4)$, $QMF_2(\Gamma_0(2))_w$ its $\Q$-subspace of weight $w$ (i.e. generated by monomials of weight smaller or equal to $w$), and $l^{QMF_2}_w$ the dimension of $QMF_2(\Gamma_0(2))_w$ as a $\Q$-vector space. Then it suffices to compute strictly more than $l^{QMF_2}_w$ terms in the development of $Z'$ to find the linear dependance between $Z'$ and the elements of $QMF_2(\Gamma_0(2))_w$ as a $\Q$-vector space. Since the developments are in powers of $q^2$, that means that we have to compute all values of the central characters for balanced partitions up to  $2(l^{QMF_2}_w+1)$. 
This method is very limited because this number grows fast and the character computations become too slow (see Table \ref{tab:l} and \ref{tab:len}).

\begin{table}[h!]
\renewcommand{\arraystretch}{1.5}
\begin{tabular}{|c|ccccc|}
\hline
$w$ & 2& 4& 6& 8& 10\\
\hline
$l^{QMF_2}_w$ &3&7&13&22&34 \\
\hline
$l^{\overline{\Lambda}}_w$ &5&20&65&185&481 \\
\hline
$l^{\Lambda^*}_w$ &2&5&11&22&42 \\
\hline
\end{tabular}\vspace{0.3cm}
\caption{Table of dimensions of $QMF_2(\Gamma_0(2)), \overline\Lambda, \Lambda^*$ as $\Q$-vector-spaces}
\label{tab:l}
\end{table}

\subsubsection{Second method}
The second method consists to apply the intermediary result of \cite{EO2} in the proof of the quasi-modularity of $Z'$ (Theorem 2 of \cite{EO2}):
\begin{theo}[Eskin--Okounkov]
The ratio $g_{\nu}(\lambda)=
\cfrac{f_{\nu, 2, \dots, 2}}{f_{2, 2,\dots, 2}}$ is the restriction of a unique function $g_{\nu}\in\overline{\Lambda}$ of weight $|\nu|/2$ to the set of balanced partitions.\end{theo}
In other words it means that the ratios $\cfrac{f_{\nu, 2, \dots, 2}}{f_{2, 2,\dots, 2}}$ are polynomials in $p_k$ and $\overline{p}_k$ of degree $|\nu|/2$, where the degree of a monomial is obtained by summing the weights $w(p_k)=k+1$ and $w(\overline{p}_k)=k$. The computation is then reduced to the computation of the $\w$-brackets of monomials in $p_k$ and $\overline p_k$ as polynomials in $E_2(q^2)$, $E_2(q^4)$, $E_4(q^4)$. 
To resume, the method here consists to:
\begin{enumerate}
\item Compute the coefficients of the polynomial $g_{\nu}$ in the $p_k$ and $\overline{p}_k$
\item Compute the coefficients of the polynomials $f_{\mu_i}$ in the $p_k$
\item Compute the coefficients of all polynomials $\langle p_{i_1}\dots p_{i_r}\overline p_{j_1}\dots \overline p_{j_s}\rangle_\w$ in $E_2(q^2)$, $E_2(q^4)$, $E_4(q^4)$.
\end{enumerate}

Note that for the first step we have to compute the values of $g_{\nu}$ on a priori at least $l^{\overline{\Lambda}}_{|\nu|/2}+1$ distinct balanced partitions, of length at least $|\nu|$, where $l^{\overline{\Lambda}}_w$ denotes the dimension of the subspace $\overline\Lambda_w$ of $\overline\Lambda$ composed by polynomials in $p_k$ and $\overline p_k$ of weight smaller than $w$ (see Table \ref{tab:l}). But another constraint appears here. Remark that \[p_1(\lambda)=|\lambda|-\cfrac{1}{24},\] and that $\overline\Lambda_{2k}$ contains the monomials $p_1^{i_1}\dots p_1^{i_s}$ with $i_1+\dots +i_s\leq k$. Each monomial $p_1^i(\lambda)$ is a polynomial in $|\lambda|$ of degree $i$. It implies that these monomials are linearly dependent  on small sets of partitions. So we have to compute the values of $g_{\nu}$ on partitions with at least $|\nu|/4+1$ distinct lengths. Note that the number of balanced partitions of length comprised between $|\nu|$ and $|\nu|+|\nu|/4$ is bigger than  $l^{\overline{\Lambda}}_w$, so it suffices to compute $g_{\nu}$ on all these partitions to obtain its coefficients. The relation on the $p_1$ is the only constraint for the choice of the set of partitions that appeared in the numerical simulations. The Table \ref{tab:len} compares the lengths of the partitions involved in the two methods, so it becomes clear that the second method is more efficient.

For the second step we can use explicit formulas given in \cite{Ls} for example, so this step presents no difficulties. Note that his conventions differs from the ones of \cite{EO2}.

Finally for the third step we have to compute $\w$-brackets of monomials in $p_k$ and $\overline p_k$, and express them in term of polynomials in $E_2(q^2)$, $E_2(q^4)$, $E_4(q^4)$. This can be done be computing the first terms of these brackets in the development in powers of $q$. Note that this time, we do not have to compute all characters, but only the dimension, and the central characters of fixed-point free involutions $f_{2, 2, \dots, 2}$, which are given by explicit formulas (equation (8) in \cite{EO2}). We noticed numerically that up to weight 12, all these $\w$-brackets are of pure weight. 

Computations giving the values of Appendix \ref{tab} were made using Pari. 

 Note that for weight 10 we used the fact that numerically the $\w$-brackets of polynomials in $p_i$, $\overline{p}_i$ are of pure weight. 

\begin{table}[h!]
\renewcommand{\arraystretch}{1.5}
\begin{tabular}{|c|ccccc|}
\hline
$w$ & 2& 4& 6& 8& 10\\
\hline
Method 1 &2-8&2-16&2-28&2-46&2-70 \\
\hline
Method 2 &4-6&8-12&12-18&16-24&20-30 \\
\hline
\end{tabular}\vspace{0.3cm}
\caption{Table of lengths of balanced partitions whose characters are computed in the two methods, for a stratum with only odd zeros ($w(\mu, \nu)=|\nu|/2$)}
\label{tab:len}
\end{table}

\subsection{Step 4: Getting the purely quadratic contribution in the case of strata with even zeros}\label{step4}
Note that for strata with only even zeros, we count covers that possibly correspond to Abelian differentials. Let $Z^\circ_{quad}(\mu, \emptyset; q)$ be the generating function for connected covers that correspond to purely quadratic differentials.
\begin{prop}
\[Z^\circ_{quad}(\mu, \emptyset; q)=Z^\circ(\mu, \emptyset; q)-2^{l(\mu)-1}Z^\circ_{ab}(\mu; q^2), \]  where $Z^\circ_{ab}(\mu; q)$ is the generating function corresponding to the stratum $\cH(\beta)$. 
\end{prop}
\begin{proof}Any connected cover of degree $2d$ of the pillow that corresponds to the square of an Abelian differential has the form
\[\pi: S\overset{\pi'}{\rightarrow}\T^2\overset{\sigma}{\rightarrow} \T^2/\pm=\fB.\]
Let $z_1, \dots, z_{l(\mu)}$ be the ramification points in $\fB$, and $z_1', \dots z_{l(\mu)}', , z_1'',\dots,  z_{l(\mu)}''$ their lift to the torus by $\sigma$. Then $\pi'$ is a ramified cover of degree $d$ of $\T^2$, ramified over ${l(\mu)}$ points $x_1, \dots x_{l(\mu)}$ with profile $\mu_i$ over $x_i$, where $x_i$ is either $z_i'$ or $z_i''$.
There are $\frac{1}{2}2^{l(\mu)}$ choices for such a $\pi'$, corresponding to the choice of the $x_i$'s: the factor $\frac{1}{2}$ is due to the symmetry induced by the double cover involution, which exchanges $(z_1', \dots , z_{l(\mu)}')$ and $(z_1'', \dots, z_{l(\mu)}'')$.
\end{proof}

\subsection{Step 5: Normalization factor between \cite{AEZ}-convention and \cite{EO2}-convention on volumes }\label{step5}

In genus 0, \cite{AEZ}-lattice points are represented by square-tiled pillowcases surfaces of equivalently by chess-colored surfaces (Lemma B.1.  in \cite{AEZ}). This is a direct consequence of the fact that in genus 0 all loops have trivial homology classes. Note that this is not anymore the case for higher genera surfaces.

\begin{lem}\label{lem:rat}We have the following normalization factor between the volumes:
\[\Vol^{AEZ}(\cQ(\alpha))=\frac{4^{\dim_\C}}{2^{l(\mu)}}\cdot \prod m_i!\cdot \Vol^{EO}(\cQ(\alpha)),\] where the $m_i$'s are the multiplicities of the odd zeros in $\alpha$, and $l(\mu)$ is the number of the even zeros.
\end{lem}
\begin{proof}
The factor $\prod m_i!$ corresponds to the labeling of the odd singularities (the even ones are labeled in the \cite{EO}-convention because they correspond to branching points over distinct points $z_1, \dots z_{l(\mu)}$ on $\fB$).

Assume first that there are no even zeros. Take a surface that corresponds to a \cite{EO}-lattice point, that is, a pillowcase cover. By convention since there are no even zeros all zeros project to the same corner of the pillow, the other three corners lift as regular points. That means that the surface is in fact tiled by squares of size $1\times 1$ (so twice larger as the pillow), so it is a {\it square-tiled pillowcase cover} (see \cite{AEZ}). On such a surface all relative cycles have holonomy in $\Z+i\Z$. In particular our preferred basis in $H^1_-(\hat S, \hat\Sigma; \C)$ have holonomy in $2\Z+2i\Z$. In other words the image of the \cite{EO}-lattice under the period map is $(2\Z+2i\Z)^{\dim_\C}$. By definition the image of the \cite{AEZ}-lattice is $(\Z+i\Z)^{\dim_\C}$ so it is clear that in this case the covolumes of the lattices are related by a factor $4^{\dim_\C}$:
\[\Covol(L^{EO})=4^{\dim_\C}\Covol(L^{AEZ}). \]

Now if there are some even zeros, let $z=(z_1, \dots, z_{l(\mu)})$ be a ${l(\mu)}$-tuple of points in $[0,1]\times[0,1/2]$. For such a $z$ we want to compare the pillowcases covers with these points as ramification points corresponding to even zeros (profile $(\mu_i,1, \dots, 1)$ over $z_i$), to the square-tiled pillowcases covers,  with profile $(\nu,2\mu, 2\dots, 2)$ over 0. The later surfaces form as previously a lattice of covolume $4^{\dim_\C}\Covol(L^{AEZ})$ in $H^1_-(\hat S, \hat \Sigma; \C)$. For the first surfaces, the holonomy of a relative cycle from an odd zero to the i-th even zero is in $\Z+i\Z\pm z_i$. It means that, for a fixed $z$, there are $2^{l(\mu)}$ more pillowcases covers with ramification over $z$ than square-tiled pillowcases covers. So we get \[\frac{1}{2^{l(\mu)}}\Vol^{EO}=\Vol^{sqp}=\frac{1}{4^{\dim_\C}}\Vol^{AEZ}.\]
\end{proof}  
Using this normalization factor, we give all volumes of strata of dimension up to 10 in the appendix \ref{tab}. Note that these values coincide with the ones computed in the previous sections.

\subsection{Conclusion}
The rationality of volumes follows from all the results of \cite{EO2}, we detail here the proof since it follows from all the detailed steps of the computation of volumes.

\begin{prop}Any stratum $\cQ(\alpha)$ of quadratic differentials has a rational Masur--Veech volume in the following sense:
\[\exists r\in\Q, \; \Vol \cQ_1(\alpha)=r\cdot \pi^{2 g_{\eff}}\]
\end{prop}

\begin{proof}
First note that the chosen normalization for the volume does not affect the result by \S\ref{step5}.
Note that for a stratum defined by partitions $\mu, \nu$, we have the following relations
 \begin{align*}
\dim_\C=2g_{\eff}+l(\mu)\\
w(\mu, \nu)=\dim_\C+l(\mu).
\end{align*} 

First the order of $Z'(\mu, \nu; q)$ as $q\to 1$ is smaller than $w(\mu,   \nu)$ by the main result of \cite{EO2}. The order of $Z^{\circ}(\mu, \nu;q)$ as $q\to 1$ is exactly $\dim_\C$. Note that if the stratum has no even zeros, the result is immediate since in this case $\dim_\C=w(\mu, \nu)=2g_{\eff}$ so only the highest order terms count in \eqref{eq:mod}, and for these terms the order of $\pi$ in the numerator coincide with the order of $h$ in the denominator.
If the stratum has $l(\mu)>0$ even zeros, in \eqref{eq:mod}, the second highest order term (in $1/h$) will be used instead of the terms in $1/h^2$, $l(\mu)$ times, such that the final order is $\dim=w-l(\mu)$. This decreases the power of $\pi$ by $2l(\mu)$ to give finally $2g_{\eff}=w-2l(\mu)$.
If the stratum has only even zeros the contribution of Abelian covers is given by $Z^\circ(\mu, q^2)$. We use the same modular transformations as \eqref{eq:mod} and an additional one for $E_6$, so the result is also true in this case.
\end{proof}

\newpage
\appendix

\section{Table of volumes}\label{tab}
\begin{center}
$\begin{array}{|c|c|c|c|}
\hline
 d & g& \textrm{Stratum} & \Vol\\
\hline

2&0& \cQ(-1^4)&2\pi^2\\ 
3&1& \cQ(2,-1^2)&4/3\pi^2\\ 
4&0& \cQ(1,-1^5)&\pi^4\\ 
4&1& \cQ(1^2,-1^2)&1/3\pi^4\\ 
4&1& \cQ(3,-1^3)&5/9\pi^4\\ 
4&2& \cQ(5,-1)&28/135\pi^4\\ 
4&2& \cQ(2^2)&2/3\pi^2\\  
5&0& \cQ(2,-1^6)&8/3\pi^4\\ 
5&1& \cQ(4,-1^4)&2\pi^4\\ 
5&1& \cQ(2,1,-1^3)&\pi^4\\ 
5&2& \cQ(6,-1^2)&184/135\pi^4\\ 
5&2& \cQ(4,1,-1)&8/15\pi^4\\ 
5&2& \cQ(2,1^2)&2/15\pi^4\\ 
5&2& \cQ(3,2,-1)&10/27\pi^4\\ 
5&3& \cQ(8)&10/27\pi^4\\ 
6&0& \cQ(1^2,-1^6)&1/2\pi^6\\ 
6&0& \cQ(3,-1^7)&3/4\pi^6\\ 
6&1& \cQ(1^3,-1^3)&11/60\pi^6\\ 
6&1& \cQ(3,1,-1^4)&1/3\pi^6\\ 
6&1& \cQ(5,-1^5)&7/10\pi^6\\ 
6&2& \cQ(1^4)&1/15\pi^6\\ 
6&2& \cQ(3,1^2,-1)&1/9\pi^6\\ 
6&2& \cQ(3^2,-1^2)&53/270\pi^6\\ 
6&2& \cQ(5,1,-1^2)&7/30\pi^6\\ 
6&2& \cQ(7,-1^3)&27/50\pi^6\\ 
6&3& \cQ(5,3)&14/243\pi^6\\ 
6&3& \cQ(7,1)&18/175\pi^6\\ 
6&3& \cQ(9,-1)&15224/42525\pi^6\\ 
6&1& \cQ(2^2,-1^4)&136/45\pi^4\\ 
6&2& \cQ(4,2,-1^2)&28/15\pi^4\\ 
6&2& \cQ(2^2,1,-1)&4/5\pi^4\\ 
6&3& \cQ(4^2)&4/5\pi^4\\ 
6&3& \cQ(6,2)&104/135\pi^4\\
7&0& \cQ(4,-1^8)&32/15\pi^6\\ 
7&0& \cQ(2,1,-1^7)&4/3\pi^6\\ 
7&1& \cQ(6,-1^6)&64/27\pi^6\\ 
7&1& \cQ(4,1,-1^5)&10/9\pi^6\\ 
7&1& \cQ(2,1^2,-1^4)&5/9\pi^6\\ 
7&1& \cQ(3,2,-1^5)&53/54\pi^6\\ 
7&2& \cQ(8,-1^4)&163/81\pi^6\\ 
7&2& \cQ(6,1,-1^3)&188/225\pi^6\\ 
7&2& \cQ(4,1^2,-1^2)&10/27\pi^6\\
7&2& \cQ(4,3,-1^3)&2/3\pi^6\\ 
7&2& \cQ(2,1^3,-1)&17/90\pi^6\\ 
7&2& \cQ(3,2,1,-1^2)&1/3\pi^6\\ 
7&2& \cQ(5,2,-1^3)&2863/4050\pi^6\\ 
7&2& \cQ(2^3, -1^2) & 256/15 \pi^4\\
7&3& \cQ(10,-1^2)&512/315\pi^6\\ 
7&3& \cQ(8,1,-1)&40/63\pi^6\\ 
\hline
\end{array}
\begin{array}{|c|c|c|c|c|}
\hline
d&g&\textrm{Stratum}&\Vol\\
\hline
7&3& \cQ(6,1^2)&232/945\pi^6\\ 
7&3& \cQ(6,3,-1)&776/1701\pi^6\\ 
7&3& \cQ(4,3,1)&32/189\pi^6\\ 
7&3& \cQ(5,4,-1)&56/135\pi^6\\ 
7&3& \cQ(3^2,2)&977/8505\pi^6\\ 
7&3& \cQ(5,2,1)&7/45\pi^6\\ 
7&3& \cQ(7,2,-1)&81/175\pi^6\\ 
7&3& \cQ(4,2^2)&4/3\pi^4\\ 
7&4& \cQ(12)&5614/6075\pi^6\\ 
8&0& \cQ(1^3,-1^7)&1/4\pi^8\\ 
8&0& \cQ(3,1,-1^8)&3/8\pi^8\\ 
8&0& \cQ(5,-1^9)&5/8\pi^8\\ 
8&0& \cQ(2^2,-1^8)&32/9\pi^6\\ 
8&1& \cQ(1^4,-1^4)&1/10\pi^8\\ 
8&1& \cQ(3,1^2,-1^5)&13/72\pi^8\\ 
8&1& \cQ(3^2,-1^6)&13/42\pi^8\\ 
8&1& \cQ(5,1,-1^6)&3/8\pi^8\\ 
8&1& \cQ(7,-1^7)&45/56\pi^8\\ 
8&1& \cQ(4,2,-1^6)& 16/5\pi^6\\
8&1& \cQ(2^2, 1, -1^5)& 104/63\pi^6\\
8&2& \cQ(1^5,-1)&29/840\pi^8\\ 
8&2& \cQ(3,1^3,-1^2)&23/378\pi^8\\ 
8&2& \cQ(3^2,1,-1^3)&104/945\pi^8\\ 
8&2& \cQ(5,1^2,-1^3)&47/360\pi^8\\ 
8&2& \cQ(5,3,-1^4)&17/72\pi^8\\ 
8&2& \cQ(7,1,-1^4)&429/1400\pi^8\\ 
8&2& \cQ(9,-1^5)&9383/12600\pi^8\\ 
8&2& \cQ(4^2, -1^4)& 396/175\pi^6\\
8&2& \cQ(6,2,-1^4)&11936/4725\pi^6\\ 
8&2& \cQ(4,2,1,-1^3)& 118/105\pi^6\\
8&2& \cQ(2^2, 1^2, -1^2)& 76/135\pi^6\\
8&2& \cQ(3,2^2, -1^3)& 190/189\pi^6\\
8&3& \cQ(3^2,1^2)&859/22680\pi^8\\ 
8&3& \cQ(3^3,-1)&4499/68040\pi^8\\ 
8&3& \cQ(5,1^3)&49/1080\pi^8\\ 
8&3& \cQ(5,3,1,-1)&17/216\pi^8\\ 
8&3& \cQ(5^2,-1^2)&421/2520\pi^8\\ 
8&3& \cQ(7,1^2,-1)&143/1400\pi^8\\ 
8&3& \cQ(7,3,-1^2)&51/280\pi^8\\ 
8&3& \cQ(9,1,-1^2)&9383/37800\pi^8\\ 
8&3& \cQ(11,-1^3)&4506281/7144200\pi^8\\ 
8&3& \cQ(6,4,-1^2)&7792/4725\pi^6\\ 
8&3& \cQ(8,2,-1^2)&3362/1701\pi^6\\ 
8&3& \cQ(4^2, 1, -1)& 32/45\pi^6\\
8&3& \cQ(6,2,1,-1)&1264/1575\pi^6\\ 
8&3& \cQ(4,2,1^2)&44/135\pi^6\\ 
8&3& \cQ(4,3,2,-1)&116/189\pi^6\\ 
8&3& \cQ(3,2^2,1)&16/63\pi^6\\ 
8&3& \cQ(5,2^2,-1)&424/675\pi^6\\ 

\hline
\end{array}$

$\begin{array}{|c|c|c|c|}
\hline
d&g&\textrm{Stratum}&\Vol\\
\hline

8&3& \cQ(2^4)&704/315\pi^4 \\
8&4& \cQ(7,5)&12/125\pi^8\\ 
8&4& \cQ(9,3)&8261/71442\pi^8\\ 
8&4& \cQ(11,1)&2197/12250\pi^8\\ 
8&4& \cQ(13,-1)&25/49\pi^8\\
8&4& \cQ(6^2)&2888/2835\pi^6\\ 
8&4& \cQ(8,4)&200/189\pi^6\\ 
8&4& \cQ(10,2)&1936/1575\pi^6\\
9&0& \cQ(6,-1^{10})&64/35\pi^8\\ 
9&0& \cQ(4,1, -1^9) & 16/15\pi^8\\
9&0& \cQ(2, 1^2, -1^8)& 2/3\pi^8\\
9&0&\cQ(3,2,-1^9) & \pi^8\\
9&1& \cQ(8,-1^8)&8/3\pi^8\\ 
9&1& \cQ(6,1,-1^7)&56/45\pi^8\\ 
9&1& \cQ(4,1^2, -1^6) & 743/1260\pi^8\\
9&1& \cQ(4,3,-1^7) & 71/72\pi^8\\
9&1& \cQ(4,1^2,-1^6)&1531/2520\pi^8\\ 
9&1& \cQ(4,3,-1^7)&19/18\pi^8\\ 
9&1& \cQ(2,1^3,-1^5)&151/504\pi^8\\ 
9&1& \cQ(3,2,1,-1^6)&529/1008\pi^8\\ 
9&1&\cQ(5,2,-1^7)&17/16\pi^8\\
9&1& \cQ(2^3,-1^6)&302/63\pi^6\\
9&2& \cQ(10,-1^6)&1408/525\pi^8\\ 
9&2& \cQ(8,1,-1^5)&835/756\pi^8\\ 
9&2& \cQ(6,1^2,-1^4)&2183/4725\pi^8\\ 
9&2& \cQ(6,3,-1^5)&935/1134\pi^8\\ 
9&2& \cQ(4,1^3,-1^3)&103/504\pi^8\\ 
9&2& \cQ(4,3,1,-1^4)&10/27\pi^8\\ 
9&2& \cQ(5,4,-1^5)&709/900\pi^8\\ 
9&2& \cQ(2,1^4,-1^2)&43/420\pi^8\\ 
9&2& \cQ(3,2,1^2,-1^3)&557/3024\pi^8\\ 
9&2& \cQ(3^2,2,-1^4)&1879/5670\pi^8\\ 
9&2& \cQ(5,2,1,-1^4)&533/1350\pi^8\\ 
9&2& \cQ(7,2,-1^5)&639/700\pi^8\\ 
9&2& \cQ(4,2^2,-1^4)&  356/105\pi^6\\ 
9&2& \cQ(2^3,1,-1^3)&178/105\pi^6 \\
9&3& \cQ(12,-1^4)&173521/72900\pi^8\\ 
9&3& \cQ(10,1,-1^3)&3392/3675\pi^8\\
9&3& \cQ(8,1^2,-1^2)&835/2268\pi^8\\ 
9&3& \cQ(8,3,-1^3)&158233/238140\pi^8\\ 
9&3& \cQ(6,1^3,-1)&209/1350\pi^8\\ 
9&3& \cQ(6,3,1,-1^2)&29/105\pi^8\\ 
9&3& \cQ(6,5,-1^3)&1439/2430\pi^8\\ 
9&3& \cQ(4,1^4)&401/5670\pi^8\\ 
9&3& \cQ(4,3,1^2,-1)&10/81\pi^8\\ 
9&3& \cQ(4,3^2,-1^2)&167/756\pi^8\\ 
9&3& \cQ(5,4,1,-1^2)&709/2700\pi^8\\ 
9&3& \cQ(7,4,-1^3)&3009/4900\pi^8\\ 
9&3& \cQ(3,2,1^3)&4/63\pi^8\\ 
\hline
\end{array}
\begin{array}{|c|c|c|c|}
\hline
d&g&\textrm{Stratum}&\Vol\\
\hline

9&3& \cQ(3^2,2,1,-1)&841/7560\pi^8\\ 
9&3& \cQ(5,2,1^2,-1)&2147/16200\pi^8\\ 
9&3& \cQ(5,3,2,-1^2)&2297/9720\pi^8\\ 
9&3& \cQ(7,2,1,-1^2)&429/1400\pi^8\\ 
9&3& \cQ(9,2,-1^3)&1788611/2381400\pi^8\\ 
9&3& \cQ(4^2,2,-1^2)&3496/1575\pi^6\\
9&3& \cQ(6,2^2,-1^2)&106112/42525\pi^6 \\
9&3& \cQ(4,2^2,1,-1)&1018/945\pi^6 \\
9&3& \cQ(2^3,1^2)&22/45\pi^6 \\
9&3& \cQ(3,2^3,-1)&1577/1701\pi^6 \\
9&4&\cQ(14,-1^2)&27560896/13395375\pi^8\\
9&4& \cQ(12,1,-1)&2639/3375\pi^8\\ 
9&4& \cQ(10,1^2)&1024/3375\pi^8\\ 
9&4& \cQ(10,3,-1)&512/945\pi^8\\ 
9&4& \cQ(8,3,1)&40/189\pi^8\\ 
9&4& \cQ(8,5,-1)&335/729\pi^8\\ 
9&4& \cQ(6,3^2)&769/5103\pi^8\\ 
9&4& \cQ(6,5,1)&619/3375\pi^8\\ 
9&4& \cQ(7,6,-1)&387/875\pi^8\\ 
9&4& \cQ(5,4,3)&11/81\pi^8\\ 
9&4& \cQ(7,4,1)&23/125\pi^8\\ 
9&4& \cQ(9,4,-1)&33814/70875\pi^8\\ 
9&4& \cQ(5^2,2)&343253/2551500\pi^8\\ 
9&4& \cQ(7,3,2)&3/20\pi^8\\ 
9&4& \cQ(9,2,1)&2959/13500\pi^8\\ 
9&4& \cQ(11,2,-1)&32141083/53581500\pi^8\\ 
9&4& \cQ(4^3)&58/45\pi^6\\
9&4& \cQ(6,4,2)& 6716/4725\pi^6 \\
9&4& \cQ(8,2^2)&791/486\pi^6\\
9&5& \cQ(16)&1042619/661500\pi^8\\
10&0& \cQ(1^4,-1^8)&1/8\pi^{10}\\ 
10&0& \cQ(3,1^2, -1^9)&3/16\pi^{10}\\
10&0& \cQ(3^2,-1^{10})&9/32\pi^{10}\\ 
10&0& \cQ(5,1,-1^{10})&5/16\pi^{10}\\ 
10&0& \cQ(7,-1^{11})&35/64\pi^{10}\\ 
10&0& \cQ(4,2,-1^{10})&128/45\pi^8\\
10&0& \cQ(2^2,1,-1^9)&16/9\pi^8\\
10&1& \cQ(1^5,-1^5)&163/3024\pi^{10}\\ 
10&1& \cQ(3,1^3,-1^6)&1159/12096\pi^{10}\\ 
10&1& \cQ(3^2,1,-1^7)&47/288\pi^{10}\\ 
10&1& \cQ(5,1^2, -1^7)&113/576\pi^{10}\\
10&1& \cQ(5,3, -1^8)&139/432\pi^{10}\\
10&1& \cQ(7,1,-1^8)&5/12\pi^{10}\\ 
10&1& \cQ(9,-1^9)&385/432\pi^{10}\\ 
10&1& \cQ(4^2,-1^8)&416/135\pi^8\\
10&1& \cQ(6,2,-1^8)& 29632/8505\pi^8\\
10&1& \cQ(4,2,1,-1^7)&682/405\pi^8\\
10&1& \cQ(2^2,1^2,-1^6)&499/567\pi^8\\
10&1& \cQ(3,2^2,-1^7)&733/486\pi^8\\

\hline
\end{array}$

$
\begin{array}{|c|c|c|c|}
\hline
d&g&\textrm{Stratum}&\Vol\\
\hline

10&2& \cQ(1^6,-1^2)&337/18144\pi^{10}\\ 
10&2& \cQ(3,1^4,-1^3)&403/12096\pi^{10}\\ 
10&2& \cQ(3^2,1^2,-1^4)&8302/136080\pi^{10}\\ 
10&2& \cQ(3^3,-1^5)&3247/30240\pi^{10}\\ 
10&2& \cQ(5,1^3,-1^4)&103/1440\pi^{10}\\ 
10&2& \cQ(5,3,1,-1^5)&37/288\pi^{10}\\ 
10&2& \cQ(5^2,-1^6)&233/864\pi^{10}\\ 
10&2& \cQ(7,1^2,-1^5)&1697/10080\pi^{10}\\ 
10&2& \cQ(7,3,-1^6)&491/1680\pi^{10}\\ 
10&2& \cQ(9,1,-1^6)&4037/10080\pi^{10}\\ 
10&2& \cQ(11,-1^7)&35113/36288\pi^{10}\\ 
10&2& \cQ(6,4,-1^6)&38432/14175\pi^8\\
10&2& \cQ(8,2,-1^6)&1838/567\pi^8\\
10&2& \cQ(4^2,1,-1^5)&17503/14175\pi^8\\
10&2& \cQ(6,2,1,-1^5)&19576/14175\pi^8 \\
10&2& \cQ(4,2,1^2,-1^4)&1501/2430\pi^8\\
10&2& \cQ(4,3,2,-1^5)&2503/2268\pi^8\\
10&2& \cQ(2^2,1^3,-1^3)&874/8505\pi^8 \\
10&2& \cQ(3,2^2,1,-1^4)& 3775/6804\pi^8\\
10&2& \cQ(5,2^2,-1^5)&85969/72900\pi^8\\
10&3& \cQ(3,1^5)&13/1134\pi^{10}\\
10&3& \cQ(3^2,1^3,-1)&16459/816480\pi^{10}\\ 
10&3& \cQ(3^3,1,-1^2)&1843/51030\pi^{10}\\ 
10&3& \cQ(5,1^4,-1)&13/540\pi^{10}\\ 
10&3& \cQ(5,3,1^2,-1^2)&167/3888\pi^{10}\\ 
10&3& \cQ(5,3^2,-1^3)&6029/777760\pi^{10}\\ 
10&3& \cQ(5^2,1,-1^3)&1859/20160\pi^{10}\\ 
10&3& \cQ(7,1^3,-1^2)&4211/75600\pi^{10}\\ 
10&3& \cQ(7,3,1,-1^3)&2027/20160\pi^{10}\\ 
10&3& \cQ(7,5,-1^4)&259/1200\pi^{10}\\ 
10&3& \cQ(9,1^2,-1^3)&372713/2721600\pi^{10}\\ 
10&3& \cQ(9,3,-1^4)&16819/68040\pi^{10}\\ 
10&3& \cQ(11,1,-1^4)&7476157/21432600\pi^{10}\\ 
10&3& \cQ(13,-1^5)&12725/14112\pi^{10}\\ 
10&3& \cQ(6^2,-1^4)&5608672/2679075\pi^8 \\
10&3& \cQ(8,4,-1^4)& 53245/23814\pi^8\\ 
10&3& \cQ(10,2,-1^4)&276352/99225\pi^8 \\
10&3& \cQ(6,4,1,-1^3)&13166/14175\pi^8 \\
10&3& \cQ(8,2,1,-1^3)&2525/2268\pi^8 \\
10&3& \cQ(4^2,1^2,-1^2)&17503/42525\pi^8\\
10&3& \cQ(6,2,1^2,-1^2)&6572/14175\pi^8 \\
10&3& \cQ(4^2,3,-1^3)&301/405\pi^8\\
10&3& \cQ(6,3,2,-1^3)&63862/76545\pi^8 \\
10&3& \cQ(4,2,1^3,-1)&2353/11340\pi^8 \\
10&3& \cQ(4,3,2,1,-1^2)&1261/3402\pi^8 \\
10&3& \cQ(5,4,2,-1^3)&19291/24300\pi^8\\ 
10&3& \cQ(3,2^2,1^2,-1)&3811/20412\pi^8 \\
10&3& \cQ(3^2,2^2,-1^2)&25517/76545\pi^8 \\
10&3& \cQ(5,2^2,1,-1^2)&9643/24300\pi^8 \\

\hline\end{array} 
\begin{array}{|c|c|c|c|}
\hline
d&g&\textrm{Stratum}&\Vol\\
\hline

10&3& \cQ(7,2^2,-1^3)&647/700\pi^8 \\
10&4& \cQ(3^4)&407867/18370800\pi^{10}\\ 
10&4& \cQ(5,3^2,1)&1541/58320\pi^{10}\\ 
10&4& \cQ(5^2,1^2)&268/8505\pi^{10}\\ 
10&4& \cQ(5^2,3,-1)&755/13608\pi^{10}\\ 
10&4& \cQ(7,3,1^2)&37/1080\pi^{10}\\ 
10&4& \cQ(7,3^2,-1)&1523/25200\pi^{10}\\ 
10&4& \cQ(7,5,1,-1)&259/3600\pi^{10}\\ 
10&4& \cQ(7^2,-1^2)&42083/252000\pi^{10}\\ 
10&4& \cQ(9,1^3)&23881/510300\pi^{10}\\ 
10&4& \cQ(9,3,1,-1)&16819/204120\pi^{10}\\ 
10&4& \cQ(9,5,-1^2)&34133/194400\pi^{10}\\ 
10&4& \cQ(11,1^2,-1)&7476157/64297800\pi^{10}\\ 
10&4& \cQ(11,3,-1^2)&32116747/154314720\pi^{10}\\ 
10&4& \cQ(13,1,-1^2)&12725/42336\pi^{10}\\ 
10&4& \cQ(15,-1^3)&3075526457/3857868000\pi^{10}\\ 
10&4& \cQ(8,6,-1^2)&59270/35721\pi^8 \\
10&4& \cQ(10,4,-1^2)&914432/496125\pi^8 \\
10&4& \cQ(12,2,-1^2)&1295123/546750\pi^8 \\
10&4& \cQ(6^2,1,-1)&1808/2625\pi^8 \\
10&4& \cQ(8,4,1,-1)&3746/5103\pi^8 \\
10&4& \cQ(10,2,1,-1)&151936/165375\pi^8 \\
10&4& \cQ(6,4,1^2)&193604/637875\pi^8 \\
10&4& \cQ(8,2,1^2)&1241/3402\pi^8 \\
10&4& \cQ(6,4,3,-1)&4636/8505\pi^8 \\
10&4& \cQ(8,3,2,-1)&233833/357210\pi^8 \\
10&4& \cQ(4^2,3,1)&400/1701\pi^8 \\
10&4& \cQ(6,3,2,1)&752/2835\pi^8 \\
10&4& \cQ(5,4^2,-1)&15596/30375\pi^8\\ 
10&4& \cQ(6,5,2,-1)&474376/820125\pi^8 \\
10&4& \cQ(4,3^2,2)&163/810\pi^8 \\
10&4& \cQ(5,4,2,1)&44617/182250\pi^8 \\
10&4& \cQ(7,4,2,-1)&7269/12250\pi^8 \\
10&4& \cQ(5,3,2^2)&6704/32805\pi^8 \\
10&4& \cQ(7,2^2,1)&727/2625\pi^8 \\
10&4& \cQ(9,2^2,-1)&28968137/40186125\pi^8 \\
10&5& \cQ(9,7)&54527/441000\pi^{10}\\ 
10&5& \cQ(11,5)&618346469/4546773000\pi^{10}\\ 
10&5& \cQ(13,3)&19615/116424\pi^{10}\\ 
10&5& \cQ(15,1)&3719141/14553000\pi^{10}\\ 
10&5& \cQ(17,-1)&2778996658/3978426375\pi^{10}\\
10&5& \cQ(8^2)&40606/32805\pi^8 \\
10&5& \cQ(10,6)&272768/212625\pi^8 \\
10&5& \cQ(12,4)&29197/20250\pi^8 \\
10&5& \cQ(14,2)&24718528/13395375\pi^8 \\
11&0& \cQ(8,-1^{12})&512/315\pi^{10}\\
11&0& \cQ(6,1,-1^{11})&32/35\pi^{10}\\
11&0& \cQ(4,1^2,-1^{10})&8/15\pi^{10} \\
11&0& \cQ(4,3,-1^{11})&4/5\pi^{10} \\

\hline
\end{array}
$
\end{center}

\begin{center}

$
\begin{array}{|c|c|c|c|}
\hline
d&g&\textrm{Stratum}&\Vol\\
\hline
11&0& \cQ(2,1^3,-1^9)&1/3\pi^{10}\\
11&0& \cQ(3,2,1,-1^{10})&1/2\pi^{10}\\
11&0& \cQ(5,2,-1^{11})&5/6\pi^{10}\\
11&1& \cQ(10,-1^{10})&512/175\pi^{10}\\ 
11&1& \cQ(8,1,-1^9)& 48/35\pi^{10}\\ 
11&1& \cQ(6,1^2,-1^8)&9136/14175\pi^{10}\\ 
11&1& \cQ(6,3,-1^9)&328/315\pi^{10}\\ 
11&1& \cQ(4,1^3,-1^7)&139/450\pi^{10}\\ 
11&1& \cQ(4,3,1,-1^8)&833/1620\pi^{10}\\ 
11&1& \cQ(5,4,-1^9)&99/100\pi^{10}\\
11&1& \cQ(2,1^4,-1^6)&67/420\pi^{10}\\ 
11&1& \cQ(3,2,1^2,-1^7)&1783/6480\pi^{10}\\ 
11&1& \cQ(3^2,2,-1^8)&80881/174960\pi^{10}\\ 
11&1& \cQ(5,2,1,-1^8)&14771/27000\pi^{10}\\
11&1& \cQ(7,2,-1^9)&901/800\pi^{10}\\ 
11&2& \cQ(12,-1^8)&34556/10125\pi^{10}\\
11&2& \cQ(10,1,-1^7)&64/45\pi^{10} \\
11&2& \cQ(8,1^2,-1^6)&13423/22680\pi^{10}\\
11&2& \cQ(8,3,-1^7)&33079/32400\pi^{10}\\ 
11&2& \cQ(6,1^3,-1^5)&21293/85050\pi^{10} \\
11&2& \cQ(6,3,1,-1^6)&75359/170100\pi^{10} \\
11&2& \cQ(6,5,-1^7)&1000111/1093500\pi^{10}\\ 
11&2& \cQ(4,1^4,-1^4)&527/4725\pi^{10}\\
11&2& \cQ(4,3,1^2,-1^5)&6829/34020\pi^{10}\\
11&2& \cQ(4,3^2,-1^6)&30049/85050\pi^{10}\\
11&2& \cQ(5,4,1,-1^6)&3427/8100\pi^{10}\\ 
11&2& \cQ(7,4,-1^7)& 2486/2625\pi^{10}\\
11&2& \cQ(2,1^5,-1^3)&301/5400\pi^{10} \\
11&2& \cQ(3,2,1^3,-1^4)&3907/38880\pi^{10} \\
11&2& \cQ(3^2,2,1,-1^5)&24487/136080\pi^{10} \\
11&2& \cQ(5,2,1^2,-1^5)&69433/324000\pi^{10}\\ 
11&2& \cQ(5,3,2,-1^6)&1986169/5248800\pi^{10}\\
11&2& \cQ(7,2,1,-1^6)&4813/9800\pi^{10}\\ 
11&2& \cQ(9,2,-1^7)&42563653/36741600\pi^{10}\\
11&3& \cQ(14,-1^6)&134276096/40186125\pi^{10} \\
11&3& \cQ(12,1,-1^5)&1417537/1093500\pi^{10} \\
11&3& \cQ(10,1^2,-1^4)&83648/165375\pi^{10} \\
11&3& \cQ(10,3,-1^5)&89888/99225\pi^{10} \\
11&3& \cQ(8,1^3,-1^3)&27457/136080\pi^{10} \\
11&3& \cQ(8,3,1,-1^4)&433967/1190700\pi^{10} \\
11&3& \cQ(8,5,-1^5)&2519/3240\pi^{10}\\ 
11&3& \cQ(6,1^4,-1^2)&851/10125\pi^{10} \\
11&3& \cQ(6,3,1^2,-1^3)&77299/510300\pi^{10} \\
11&3& \cQ(6,3^2,-1^4)&104486/382725\pi^{10} \\
11&3& \cQ(6,5,1,-1^4)&29632/91125\pi^{10} \\
11&3& \cQ(7,6,-1^5)&13877/18375\pi^{10} \\
11&3& \cQ(4,1^5,-1)&79/2100\pi^{10} \\
11&3& \cQ(4,3,1^3,-1^2)&457/6804\pi^{10} \\
11&3& \cQ(4,3^2,1,-1^3)&442/3645\pi^{10} \\
\hline
\end{array}
\begin{array}{|c|c|c|c|}
\hline
d&g&\textrm{Stratum}&\Vol\\
\hline

11&3& \cQ(5,4,1^2,-1^3)&2519/3240\pi^{10}\\ 
11&3& \cQ(5,4,3,-1^4)&6323/24300*\pi^{10}\\ 
11&3& \cQ(7,4,1,-1^4)& 14881/44100\pi^{10}\\
11&3& \cQ(9,4,-1^5)&4908079/5953500\pi^{10}\\ 
11&3& \cQ(2,1^6)&4343/226800\pi^{10} \\
11&3& \cQ(3,2,1^4,-1)&3067/90720\pi^{10} \\
11&3& \cQ(3^2,2,1^2,-1^2)&6173/102060\pi^{10} \\
11&3& \cQ(3^3,2,-1^3)&5957/54675\pi^{10} \\
11&3& \cQ(5,2,1^3,-1^2)&209969/2916000\pi^{10} \\
11&3& \cQ(5,3,2,1,-1^3)&37823/291600\pi^{10} \\
11&3& \cQ(5^2,2,-1^4)&946247/3402000\pi^{10} \\
11&3& \cQ(7,2,1^2,-1^3)&14123/84000\pi^{10} \\
11&3& \cQ(7,3,2,-1^4)&53/175\pi^{10} \\
11&3& \cQ(9,2,1,-1^4)&5902831/14288400\pi^{10} \\
11&3& \cQ(11,2,-1^5)&1341201979/1285956000\pi^{10} \\
11&4& \cQ(16,-1^4)&30373/10000\pi^{10} \\
11&4& \cQ(14,1,-1^3)&76128992/66976875\pi^{10} \\
11&4& \cQ(12,1^2,-1^2)&1417537/3280500\pi^{10} \\
11&4& \cQ(12,3,-1^3)&1534159/1968300\pi^{10} \\
11&4& \cQ(10,1^3,-1)&3104/18375\pi^{10} \\
11&4& \cQ(10,3,1,-1^2)&4288/14175\pi^{10} \\
11&4& \cQ(10,5,-1^3)&32864/50625\pi^{10} \\
11&4& \cQ(8,1^4)&2339/34020\pi^{10} \\
11&4& \cQ(8,3,1^2,-1)&433967/3572100\pi^{10} \\
11&4& \cQ(8,3^2,-1^2)&23341721/107163000\pi^{10} \\
11&4& \cQ(8,5,1,-1^2)&2519/9720\pi^{10} \\
11&4& \cQ(8,7,-1^3)&5939/9800\pi^{10} \\
11&4& \cQ(6,3,1^3)&6568/127575\pi^{10} \\
11&4& \cQ(6,3^2,1,-1)&116231/1275750\pi^{10} \\
11&4& \cQ(6,5,1^2,-1)&59333/546750\pi^{10} \\
11&4& \cQ(6,5,3,-1^2)&63767/328050\pi^{10} \\
11&4& \cQ(7,6,1,-1^2)&138931/551250\pi^{10} \\
11&4& \cQ(9,6,-1^3)&248951329/401861250\pi^{10} \\
11&4& \cQ(4,3^2,1^2)&62987/1530900\pi^{10} \\
11&4& \cQ(4,3^3,-1)&37193/510300\pi^{10} \\
11&4& \cQ(5,4,1^3)&5963/121500\pi^{10} \\
11&4& \cQ(5,4,3,1,-1)&6323/72900\pi^{10} \\
11&4& \cQ(5^2,4,-1^2)&157357/850500\pi^{10}\\
11&4& \cQ(7,4,1^2,-1)&14881/132300\pi^{10} \\
11&4& \cQ(7,4,3,-1^2)&8891/44100\pi^{10} \\
11&4& \cQ(9,4,1,-1^2)&4908079/17860500\pi^{10} \\
11&4& \cQ(11,4,-1^3)&187259839/267907500\pi^{10} \\
11&4& \cQ(3^3,2,1)&37859/1020600\pi^{10} \\
11&4& \cQ(5,3,2,1^2)&4831/109350\pi^{10} \\
11&4& \cQ(5,3^2,2,-1)&2050399/26244000\pi^{10} \\
11&4& \cQ(5^2,2,1,-1)&1898809/20412000\pi^{10} \\
11&4& \cQ(7,2,1^3)&1547/27000\pi^{10} \\
11&4& \cQ(7,3,2,1,-1)&1021/10080\pi^{10} \\
11&4& \cQ(7,5,2,-1^2)&19429/90000\pi^{10} \\
\hline
\end{array}
$
\end{center}

$
\begin{array}{|c|c|c|c|}
\hline
d&g&\textrm{Stratum}&\Vol\\
\hline

11&4& \cQ(9,2,1^2,-1)&59225353/428652000\pi^{10} \\
11&4& \cQ(9,3,2,-1^2)&15901567/64297800\pi^{10} \\
11&4& \cQ(11,2,1,-1^2)&7476157/21432600\pi^{10} \\
11&4& \cQ(13,2,-1^3)&12805/14112\pi^{10} \\
11&5& \cQ(18,-1^2)&10811701157/3978426375\pi^{10}\\
11&5& \cQ(16,1,-1)&1019257/1010625\pi^{10} \\
11&5& \cQ(14,1^2)&844358464/2210236875\pi^{10} \\
11&5& \cQ(14,3,-1)&540294976/795685275\pi^{10} \\
11&5& \cQ(12,3,1)&52108/200475\pi^{10} \\
11&5& \cQ(12,5,-1)&1571104223/2841733125\pi^{10} \\
11&5& \cQ(10,3^2)&1261376/7016625\pi^{10} \\
11&5& \cQ(10,5,1)&838592/3898125\pi^{10} \\
11&5& \cQ(10,7,-1)&72256/144375\pi^{10} \\
11&5& \cQ(8,5,3)&11532707/75779550\pi^{10} \\
11&5& \cQ(8,7,1)&767/3850\pi^{10} \\
11&5& \cQ(9,8,-1)&1687928/3444525\pi^{10} \\
11&5& \cQ(6,5^2)&126494329/947244375\pi^{10} \\
11&5& \cQ(7,6,3)&2531/17325\pi^{10} \\
11&5& \cQ(9,6,1)&645007/3189375\pi^{10} \\
11&5& \cQ(11,6,-1)&10284958151/19892131875\pi^{10} \\
11&5& \cQ(7,5,4)&58876/433125\pi^{10} \\
11&5& \cQ(9,4,3)&84271/535815\pi^{10} \\
11&5& \cQ(11,4,1)&500315647/2210236875\pi^{10} \\
11&5& \cQ(13,4,-1)&2890/4851\pi^{10} \\
11&5& \cQ(7^2,2)&1504721/9702000\pi^{10} \\
11&5& \cQ(9,5,2)&226006693/1377810000\pi^{10} \\
11&5& \cQ(11,3,2)&502537061/2546192880\pi^{10} \\
11&5& \cQ(13,2,1)&68105/232848\pi^{10} \\
11&5& \cQ(15,2,-1)&50098313761/63654822000\pi^{10} \\
11&6& \cQ(20)&34148209117/14467005000\pi^{10}\\
\hline
\end{array}
$

\section{Alternative computation of volume}\label{app:alt}
We give here an alternative computation of the volume of the hyperelliptic stratum $\cQ(2,1^2)$ (\S~\ref{sect:volhyp}), using the method of diagrams couting described in \S~\ref{ssection:voldim5}. This allows us to check one more time that our choices of normalization for the volumes are consistent. Furthermore this example is a good representative of the general case: it illustrates how multi-zeta values appear in the computations and how they compensate to give at the end the expected value for the total volumes of the stratum.

\subsection{$\cQ(2, -1^2)$}\label{subsection:alt2-1-1}
The diagrams for this stratum are given in Figure \ref{fig:Q2ii}.

\begin{figure}[h!]
\begin{center}
\begin{tabular}{|c|p{2cm}|c|c|c|}
\hline
\multicolumn{2}{|c|}{Diagrams}& $(l_1, \dots, l_k)$ & Sym & Contribution \\
\hline
\includegraphics{plot/oneloop.eps}
\begin{picture}(50,20)
\put(7,-20)
{\begin{picture}(0,0)
\put(3,13){\tiny{0}}
\put(14,24){$w$}
\end{picture}}\end{picture}
&
\includegraphics{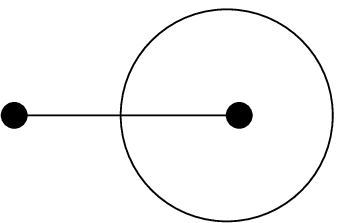}

& $w$ & 1 & $\cfrac{4N^3}{3}\zeta(3)$\\
&&&&  \\
\hline
\includegraphics{plot/ballon.eps}

\begin{picture}(50,70)
\put(7,-20)
{\begin{picture}(0,0)
\put(14,2){\tiny{0}}
\put(20, 20){$2w$}
\put(13, 65){$w$}
\put(14,42){\tiny{0}}
\end{picture}}\end{picture}
&
\includegraphics{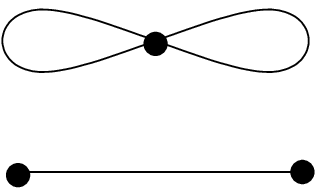}
& 1&$\cfrac{1}{2}$ &$ \cfrac{N^3}{6}(8\zeta(2)-9\zeta(3))$ \\
&&&&\\
\hline
\includegraphics{plot/tige.eps}

\begin{picture}(50,50)
\put(16,-25)
{\begin{picture}(0,0)
\put(3,7){\tiny{0}}
\put(3,46){\tiny{1}}
\put(8,25){$W$}
\end{picture}}\end{picture}
&
\includegraphics{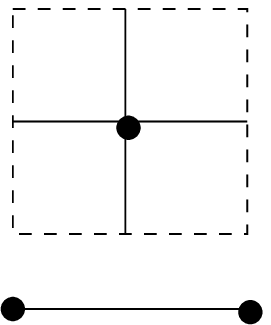}
& $W$ & $\cfrac{1}{2}$ & $ \cfrac{N^3}{6}\zeta(3)$\\
&&&&\\
\hline
\end{tabular}
\end{center}
\caption{
\label{fig:Q2ii}
Diagrams for $\cQ(2, -1^2)$
}
\end{figure}

Summing all the contributions we get $\cfrac{4N^3}{3}\zeta(2)$ so by (\ref{eq:volint}), we obtain:\[\Vol \cQ(2, -1^2)=8\zeta(2)=\cfrac{4\pi^2}{3},\]
which coincides with the value found in (\ref{eq:ex2}).

\subsection{$\cQ(1^2, -1^2)$}\label{subsection:alt11-1-1}
The diagrams for this stratum are given in Figure \ref{fig:Q11ii}.

\begin{figure}[h!]
\begin{center}
\begin{tabular}{|c|p{2cm}|c|c|c|}
\hline
\multicolumn{2}{|c|}{Diagrams}& $(l_1, \dots, l_k)$ & weight & Contribution \\
\hline
\includegraphics{plot/oneloop.eps}
\begin{picture}(50,20)
\put(7,-20)
{\begin{picture}(0,0)
\put(3,13){\tiny{0}}
\put(14,24){$w$}
\end{picture}}\end{picture}
&
\includegraphics{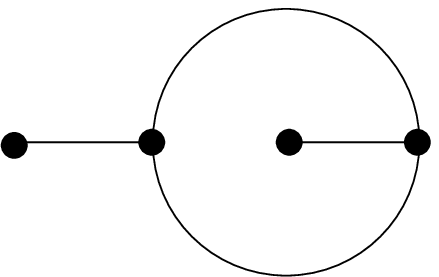}
& $w^2$ & 2 & $ \cfrac{1}{2}\cfrac{(2N)^4}{4}\zeta(4)=2 N^4\zeta(4)$ \\
&&&&\\
\hline
\includegraphics{plot/ballon.eps}

\begin{picture}(50,60)
\put(7,-30)
{\begin{picture}(0,0)
\put(14,2){\tiny{0}}
\put(20, 20){$W_1$}
\put(13, 65){$w_2$}
\put(14,42){\tiny{0}}
\end{picture}}\end{picture}
&
\includegraphics{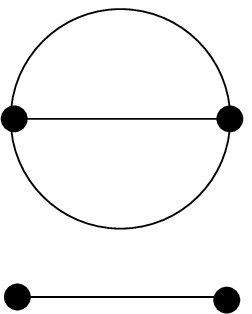}
&$\mathds 1_{\{2w_2>W_1\}}$ & $\cfrac{1}{3}{3\choose 1}$&\\
\cline{2-4}
\begin{picture}(50,20)
\end{picture}
&
\includegraphics{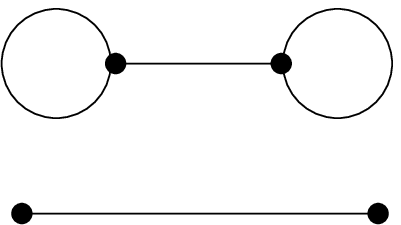}
&$\mathds 1_{\{2w_2<W_1\}}$ & $\cfrac{1}{3}{3\choose 1}$&$ \cfrac{N^4}{3}(\zeta(2))^2$
\\
&&&&\\
\hline
\includegraphics{plot/tige.eps}

\begin{picture}(50,50)
\put(16,-25)
{\begin{picture}(0,0)
\put(3,7){\tiny{0}}
\put(3,46){\tiny{1}}
\put(8,25){$W$}
\end{picture}}\end{picture}
&
\includegraphics{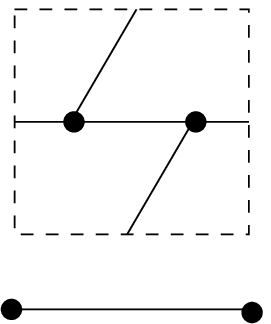}
&$\cfrac{W^2}{2}$& $\cfrac{1}{3}$&$ \cfrac{N^4}{12}\zeta(4)$\\
&&&&\\
\hline
\includegraphics{plot/duo.eps}

\begin{picture}(50,30)
\put(10,-25)
{\begin{picture}(0,0)
\put(11,40){\tiny{0}}
\put(11,7){\tiny{0}}
\put(25,20){$w_2$}
\put(-11,20){$w_1$}
\end{picture}}\end{picture}

&
\includegraphics{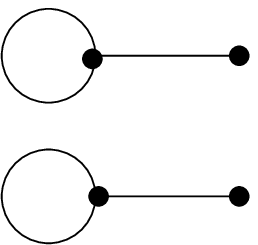}
& $\mathds 1_{\{w_1<w_2\}}\mathds 1_{\{w_2-w_1\in\N\}}$ & 2 & 
$\cfrac{5N^4}{6}\zeta(4)$\\
&&&&\\
\hline

\end{tabular}
\end{center}
\caption{
\label{fig:Q11ii}
Diagrams for $\cQ(1^2, -1^2)$
}
\end{figure}

Summing all the contributions, we obtain $\cfrac{15N^4}{4}\zeta(4)$. So by (\ref{eq:volint}):
\[\Vol\cQ(1^2, -1^2)=30\zeta(4)=\cfrac{\pi^4}{3},\]
which coincides with the value found in (\ref{eq:ex1}).

\subsection{$\cQ(2, 1^2)$ \label{subsection:alt211}}

The diagrams for this stratum are given in Figure \ref{fig:Q211}.

\begin{figure}
\begin{center}
\begin{tabular}{|c|p{2cm}|c|c|c|}
\hline
\multicolumn{2}{|c|}{Diagrams}& $(l_1, \dots, l_k)$ & weight & Contribution \\
\hline
\includegraphics{plot/twoloops.eps}

\begin{picture}(50,20)
\put(8,-18)
{\begin{picture}(0,0)
\put(14,2){\tiny{0}}
\put(20, 22){$w_1$}
\put(0, 22){$w_2$}

\end{picture}}\end{picture}
&
\includegraphics{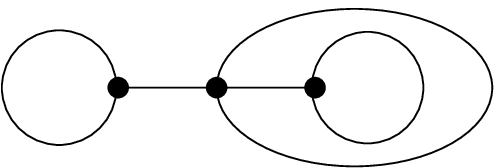}
& $(w_2-w_1)\mathds 1_{\{w_2>w_1\}}$&1& $ \displaystyle\frac{4N^5}{15}\zeta(2)\zeta(3)$ \\
\cline{2-5}
\begin{picture}(50,15)
\end{picture}
&
\includegraphics{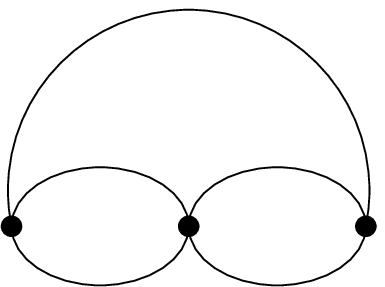}
&$\mathds 1_{\{w_2>w_1\}}$ &1 &$O(N^4)$\\
&&&&\\
\hline
\includegraphics{plot/oneloop.eps}
\begin{picture}(50,25)
\put(7,-20)
{\begin{picture}(0,0)
\put(3,13){\tiny{1}}
\put(14,24){$w$}
\end{picture}}\end{picture}
&
\includegraphics{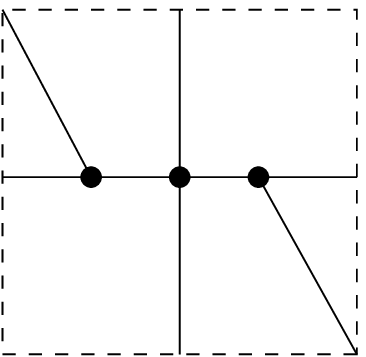}
& $\displaystyle\frac{(2w)^3}{2\cdot 3!}$ & 1 &$\displaystyle\frac{8N^5}{15}\zeta(5) $\\
&&&&\\
\hline
\includegraphics{plot/ballon.eps}

\begin{picture}(50,60)
\put(7,-30)
{\begin{picture}(0,0)
\put(14,2){\tiny{1}}
\put(20, 20){$W_1$}
\put(13, 65){$w_2$}
\put(14,42){\tiny{0}}
\end{picture}}\end{picture}
&
\includegraphics{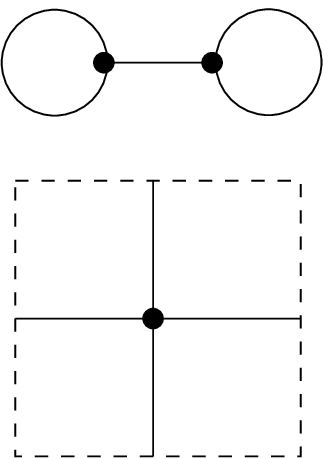}
&$W_1\mathds 1_{\{W_1>2w_2\}}$ &$\displaystyle\frac{1}{4}$ &\\
\cline{2-4}
\begin{picture}(50,50)
\end{picture}
&
\includegraphics{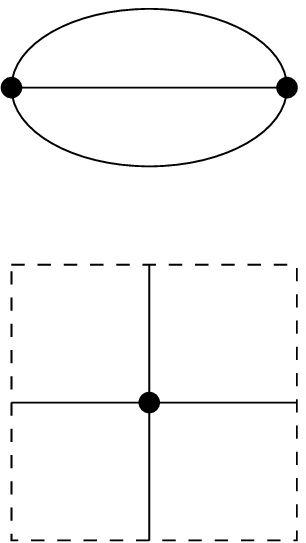}
&$W_1\mathds 1_{\{W_1<2w_2\}}$&$\displaystyle\frac{1}{4}$&$\displaystyle\frac{N^5}{30}\zeta(2)\zeta(3)$\\
&&&&\\
\hline
\includegraphics{plot/ballon.eps}
\begin{picture}(50,60)
\put(7,-18)
{\begin{picture}(0,0)
\put(14,2){\tiny{0}}
\put(20, 20){$W_1$}
\put(13, 65){$w_2$}
\put(14,42){\tiny{0}}
\end{picture}}\end{picture}
&
\includegraphics{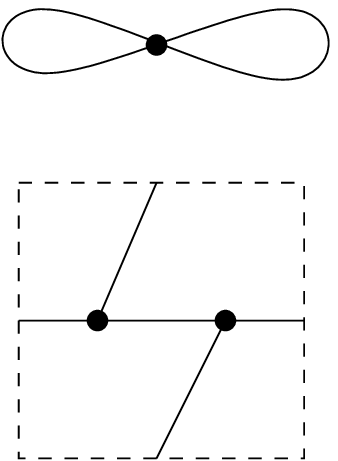}
& $\displaystyle\frac{W_1^2}{2}\mathds 1_{\{W_1=2w_2\}}$ & $\displaystyle\frac{1}{2}\cdot\displaystyle\frac{1}{3}$ &$\displaystyle\frac{N^5}{60}(32\zeta(4)-33\zeta(5))$ \\
&&&&\\
\hline
\includegraphics{plot/tige.eps}

\begin{picture}(50,60)
\put(16,-15)
{\begin{picture}(0,0)
\put(3,7){\tiny{1}}
\put(3,46){\tiny{1}}
\put(8,25){$W$}
\end{picture}}\end{picture}
&
\includegraphics{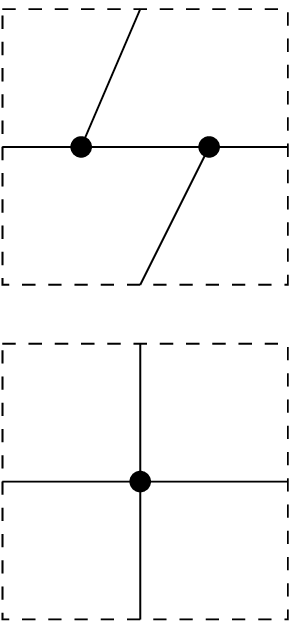}
& $W\displaystyle\frac{W^2}{2}$ & $\displaystyle\frac{1}{3}\cdot\displaystyle\frac{1}{4}$ &$\displaystyle\frac{N^5}{60}\zeta(5)$ \\
&&&&\\
\hline
\includegraphics{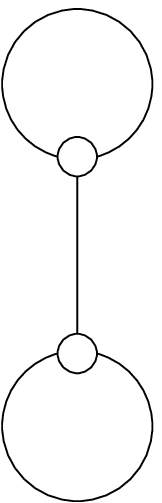}

\begin{picture}(50,50)
\put(9,-28)
{\begin{picture}(0,0)
\put(14,3){\tiny{0}}
\put(20, 20){$W_2$}
\put(13, 65){$w_3$}
\put(14,42){\tiny{0}}
\put(13, -20){$w_1$}
\end{picture}}\end{picture}
&
\includegraphics{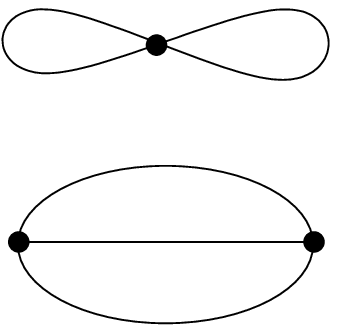}
& $\mathds 1_{\{w_1>w_3, W_2=2w_1\}}$ & $\displaystyle\frac{1}{2}$ &\\
\cline{2-4}
\begin{picture}(50,20)
\end{picture}
&
\includegraphics{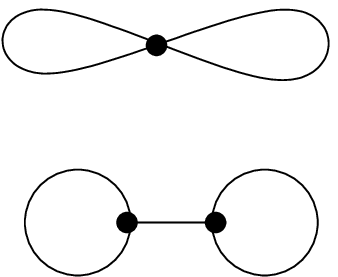}
& $\mathds 1_{\{w_3>w_1, W_2=2w_1\}}$ & $\displaystyle\frac{1}{2}$ &$\displaystyle\frac{N^5}{30}(8\zeta(2)^2-9\zeta(2)\zeta(3))$\\
&&&&\\
\hline

\end{tabular}
\end{center}
\caption{
\label{fig:Q211}
Diagrams for $\cQ(2, 1^2)$
}
\end{figure}

Summing all the contributions we get $\cfrac{6N^5}{5}\zeta(4)$ so by (\ref{eq:volint}), we obtain:\[\Vol \cQ(2, 1^2)=12\zeta(4)=\cfrac{2\pi^4}{15},\]
which coincides with the value found in (\ref{eq:ex2}).

\subsection{$\cQ(2,2)$ \label{subsection:alt22}}
On the 8 configurations of ribbon graphs of genus 2 with two vertices of valency 4, only 5 correspond to flat surfaces in the stratum $\cQ(2,2)$ (see Figure \ref{fig:Q22}), the others belong to $\cH(1,1)$ (see Figure \ref{fig:H11}).

\begin{figure}[h!]
\begin{center}
\begin{tabular}{|c|p{2cm}|c|c|c|}
\hline
\multicolumn{2}{|c|}{Diagrams}& $(l_1, \dots, l_k)$ & weight & Contribution \\
\hline
\includegraphics{plot/twoloops.eps}

\begin{picture}(50,30)
\put(8,-5)
{\begin{picture}(0,0)
\put(14,2){\tiny{0}}
\put(20, 22){$w_1$}
\put(0, 22){$w_2$}
\put(84,2){2}
\put(84,-8){2}
\put(84,15){1}
\put(84,25){1}

\end{picture}}\end{picture}
&
\includegraphics{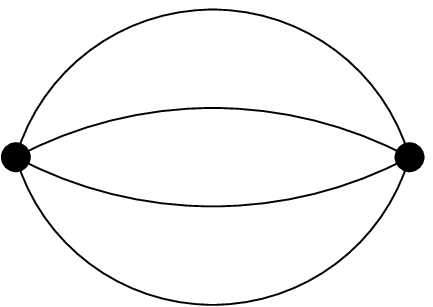}
& $2w_1\mathds 1_{\{w_2=w_1\}}$&$\displaystyle\frac{1}{4}$& $ N^4(\zeta(3)-\zeta(4))$ \\
&&&&\\
\hline
\includegraphics{plot/oneloop.eps}
\begin{picture}(50,20)
\put(7,-20)
{\begin{picture}(0,0)
\put(3,13){\tiny{0}}
\put(14,24){$w$}
\end{picture}}\end{picture}
&
\includegraphics{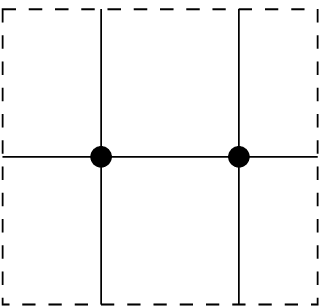}
& $\displaystyle\frac{(2w)^2}{2\cdot 2}$ & $\displaystyle\frac{1}{2} $&$\displaystyle\frac{N^4}{2}\zeta(4) $\\
&&&&\\
\hline
\includegraphics{plot/ballon.eps}
\begin{picture}(50,60)
\put(7,-18)
{\begin{picture}(0,0)
\put(14,2){\tiny{0}}
\put(20, 20){$W_1$}
\put(13, 65){$w_2$}
\put(14,42){\tiny{0}}
\end{picture}}\end{picture}
&
\includegraphics{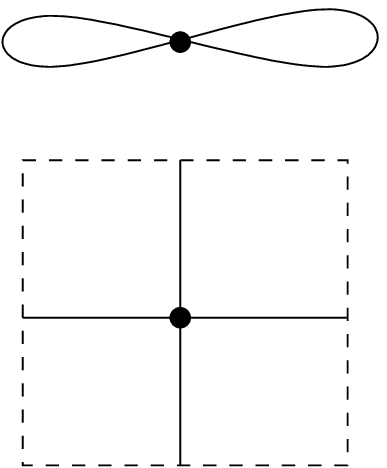}
& $W_1\mathds 1_{\{W_1=2w_2\}}$ & $2\cdot\displaystyle\frac{1}{4}\cdot\displaystyle\frac{1}{2}$ &$\displaystyle\frac{N^4}{16}(16\zeta(3)-17\zeta(4))$ \\
&&&&\\
\hline
\includegraphics{plot/tige.eps}

\begin{picture}(50,60)
\put(16,-15)
{\begin{picture}(0,0)
\put(3,7){\tiny{1}}
\put(3,46){\tiny{1}}
\put(8,25){$W$}
\end{picture}}\end{picture}
&
\includegraphics{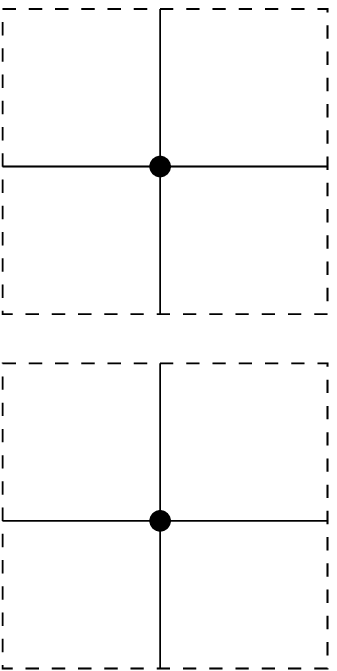}
& $W^2$ & $\displaystyle\frac{1}{4}\cdot\displaystyle\frac{1}{4}$ &$\displaystyle\frac{N^4}{32}\zeta(4)$ \\
&&&&\\
\hline
\includegraphics{plot/halteres.eps}

\begin{picture}(50,90)
\put(9,7)
{\begin{picture}(0,0)
\put(14,3){\tiny{0}}
\put(20, 20){$W_2$}
\put(13, 65){$w_3$}
\put(14,42){\tiny{0}}
\put(13, -20){$w_1$}
\end{picture}}\end{picture}

&
\includegraphics{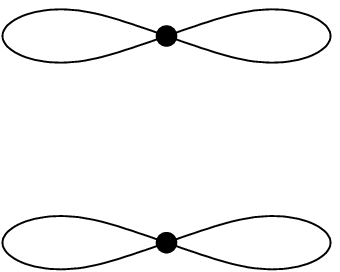}
& $\mathds 1_{\{W_2=2w_1=2w_3\}}$ & $\displaystyle\frac{1}{2}\cdot\displaystyle\frac{1}{2}$ &$\displaystyle\frac{N^4}{2}(\zeta(2)-4\zeta(3)+\displaystyle\frac{49}{16}\zeta(4))$\\
&&&&\\
\hline

\end{tabular}
\end{center}
\caption{
\label{fig:Q22}
Diagrams for $\cQ(2, 2)$
}
\end{figure}

Summing all the contributions we get $\cfrac{N^4}{2}\zeta(2)$ so by (\ref{eq:volint}), we obtain:\[\Vol \cQ(2, 2)=4\zeta(2)=\cfrac{2\pi^2}{3},\]
which coincides with the value (\ref{eq:ex2}) divided by 2, because here we count surfaces modulo the hyperelliptic involution.

\begin{figure}[h!]
\begin{center}
\begin{tabular}{|c|p{4.5cm}|}
\hline
\includegraphics{plot/twoloops.eps}

\begin{picture}(50,40)
\put(8,5)
{\begin{picture}(0,0)
\put(14,3){\tiny{0}}
\put(84,2){2}
\put(84,-8){1}
\put(84,15){1}
\put(84,25){2}

\end{picture}}\end{picture}
&
\includegraphics{plot_autre/24genre0.eps}
\includegraphics{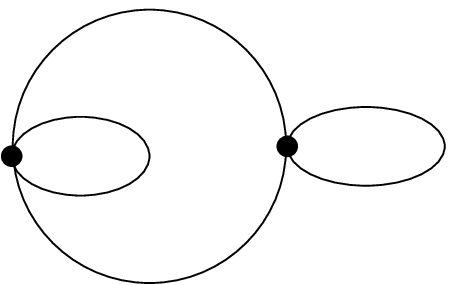}
\\
\hline
\includegraphics{plot/oneloop.eps}
\begin{picture}(50,40)
\put(7,-5)
{\begin{picture}(0,0)
\put(3,13){\tiny{0}}
\end{picture}}\end{picture}
&
\includegraphics{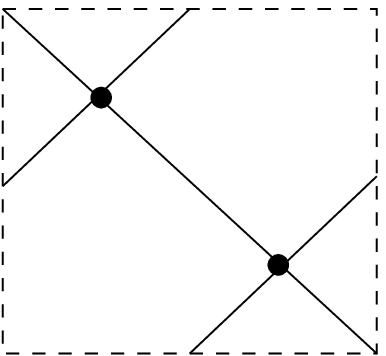}

\\
\hline
\includegraphics{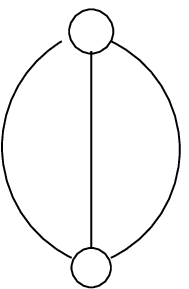}

\begin{picture}(50,55)
\put(12,-8)
{\begin{picture}(0,0)
\put(14,2){\tiny{0}}
\put(14,49){\tiny{0}}
\end{picture}}\end{picture}

&
\includegraphics{plot_autre/doublehuit.eps}
\\
\hline

\end{tabular}
\end{center}
\caption{
\label{fig:H11}
Diagrams for $\cH(1, 1)$
}
\end{figure}

\section{Toolbox}\label{app:toolbox}

 Recall that \[\zeta(2)=\cfrac{\pi^2}{6},\quad \zeta(4)=\cfrac{\pi^4}{90}\quad \mbox{so }(\zeta(2))^2=\cfrac{5}{2}\zeta(4).\]
 Recall the definition of the multiple zeta functions:\[\zeta(s_1, \dots, s_k)=\sum_{n_1>\dots>n_k>0}\frac{1}{n_1^{s_1}\dots n_k^{s_k}}\]

\begin{lem}\label{lem:tool1}
\begin{eqnarray} \forall m\geq 2, \quad \sum_{k\geq 0}\frac{1}{(2k+1)^m}=\frac{2^m-1}{2^m}\zeta(m)\label{eq:sumodd}\\
\forall m\geq 1, \quad\sum_{i=1}^{N}i^m\equi{N\to\infty} \cfrac{N^{m+1}}{m+1}\label{eq:sumpower}\\
\forall m\geq 1, \quad \Card\{(l_1, \dots, l_m)\in\N^m | N=2l_1+\dots + 2l_j+l_{j+1}+\dots +l_m \}\notag\\
\equi{N\to\infty}\cfrac{N^{m-1}}{2^j(m-1)!}\label{eq:partition}
\end{eqnarray}
\end{lem}

We recall the following standard fact (Lemma 3.7 of \cite{AEZ2}):

\begin{lem}[Athreya-Eskin-Zorich] 
\[\sum_{\substack{H\cdot W\leq N\\W\in \N^k, W \in\N^k}}W_1^{a_1+1}\dots W_k^{a_k+1}\sim \frac{N^{a+2k}}{(a+2k)!}\cdot\prod_{i=1}^k (a_i+1)\zeta(a_i+2),\]
where $a=\sum_{i=1}^k a_i$.
\end{lem}

We will need the following variations of the previous lemma: 

\begin{lem}\label{lem:tool2}
\begin{eqnarray}
\sum_{W(H_1+2H_2)\leq 2N }W^m & \sim & \frac{N^{m+1}}{2(m+1)} \left(2^{m+1}\zeta(m)-(2^{m+1}+1)\zeta(m+1)\right)\label{eq:combi1}
\\
\sum_{W(H_1+2H_2+H_3)\leq N}W^3 & \sim & \frac{N^4}{16}\left(\zeta(2)-4\zeta(3)+\frac{49}{16}\zeta(4)\right)\label{eq:combi2}\\
\sum_{\substack{W_1(H_1+2H_2)\\+W_2H_3\leq 2N}} W_1^2W_2 & \sim &  \frac{N^ 5}{30}(8(\zeta(2))^2 -9\zeta(2)\zeta(3))\label{eq:combi3}
\end{eqnarray}
\end{lem}

\begin{proof}  Proof of (\ref{eq:combi1}):
$$A=\sum_{W(H_1+2H_2)\leq 2N}W^m=\sum_{WH\leq 2N} W^m\Card\{(H_1,H_2)\in\N^2\; \textrm{s.t.}\; H=H_1+2H_2\}$$
Since $2H_2$ is even and goes from $2$ to $H-1$ or $H-2$ depending on the parity of $H$, we have : $$\Card\{(H_1,H_2)\; \textrm{s.t.}\; H=H_1+2H_2\}=\lfloor\frac{H-1}{2}\rfloor.$$
\begin{eqnarray*}A &\sim & \sum_{WH\leq 2N}W^m\lfloor\frac{H-1}{2}\rfloor  = \sum_{W(2K+1)\leq 2N}W^mK+\sum_{W(2K+2)\leq 2N}W^mK\\
& \sim &\sum_{K\geq 1}K\left(\frac{1}{m+1}\left(\frac{2N}{2K+1}\right)^{m+1}+\frac{1}{m+1}\left(\frac{2N}{2K+2}\right)^{m+1}\right)\end{eqnarray*}
using (\ref{eq:sumpower}). So

\[A=\frac{N^{m+1}}{m+1}\left(2^{m+1}\underbrace{\sum_{K\geq 0} \frac{K}{(2K+1)^{m+1}}}_{S_1(m)}+\underbrace{\sum_{K\geq 0}\frac{K}{(K+1)^{m+1}}}_{S_2(m)}\right)\]

\[2S_1(m)+\sum_{K\geq 0}\frac{1}{(2K+1)^{m+1}}=\sum_{K\geq 0}\frac{1}{(2K+1)^m}\]
So using (\ref{eq:sumodd}) we obtain: 
\[S_1(m)=\frac{1}{2^{m+2}}((2^{m+1}-2)\zeta(m)-(2^{m+1}-1)\zeta(m+1))\]
Similarly, \[S_2(m)=\zeta(m)-\zeta(m+1),\]
which gives the result.

Proof of (\ref{eq:combi2}): $$B=\sum_{W(H_1+2H_2+H_3)\leq N}W^3=\sum_{WH\leq N}W^3\Card\{(H_1,H_2,H_3)\; \textrm{s.t.}\; H=H_1+2H_2+H_3\}$$
Since $2H_2$ is even and goes from $2$ to $H-2$ or $H-3$ depending on the parity of $H$, and $H_1$ is an integer which goes from $1$ to $H-2H_2-1$, we have: $$
\Card\{(H_1,H_2,H_3)\; \textrm{s.t.}\; H=H_1+2H_2+H_3\}=\begin{cases}K(K+1)
 & \mbox{if }  H=2K+3\; (K\geq 1)\\K^2 & \mbox{if } H=2K+2\; (K\geq 1)\end{cases}$$
 So $$B\sim \frac{N^4}{4}\left(\underbrace{\sum_{K\geq 0} \frac{K(K+1)}{(2K+3)^4}}_{S_3}+\underbrace{\sum_{K\geq 0}\frac{K^2}{(2K+2)^4}}_{S_4}\right).$$
 $$S_3=\frac{1}{4}\sum_{K\geq 0}\frac{1}{(2K+3)^2}-\sum_{K\geq 0}\frac{1}{(2K+3)^3}+\frac{3}{4}\sum_{K\geq 0}\frac{1}{(2K+3)^4}$$ so by (\ref{eq:sumodd}) we have:
 \[S_3=\frac{3}{16}\zeta(2)-\frac{7}{8}\zeta(3)+\frac{45}{64}\zeta(4)\]
Similarly $$S_4=\frac{1}{2^4}(\zeta(2)-2\zeta(3)+\zeta(4)),$$ which gives the result.

Proof of (\ref{eq:combi3}):
As for (\ref{eq:combi1}) we have: \[\sum_{\substack{W_1(H_1+2h_2)\\+W_2H_3\leq 2N}} W_1^2W_2=\sum_{\substack{W_1(2K+1)\\+W_2H_3\leq 2N}}W_1^2W_2K+\sum_{\substack{W_1(2K+2)\\+W_2H_3\leq 2N}}W_1^2W_2K\]
Following the proof of Lemma 3.7 in \cite{AEZ2}, we introduce $x_1=\cfrac{W_1(2K+1)}{2N}$ and $x_2=\cfrac{W_2H_3}{2N}$. We obtain for the first sum
\begin{eqnarray*}\sum_{\substack{W_1(2K+1)\\+W_2H_3\leq 2N}}W_1^2W_2K & \sim & \sum_{\substack{K\geq 0,\\ H\geq 1}}\int_{\Delta^2}K\left(\frac{x_12N}{2K+1}\right)^2\left(\frac{x_22N}{H}\right)\frac{2N}{2K+1}dx_1\frac{2N}{H}dx_2\\
& = & (2N)^5\int_{\Delta^2}x_1^2x_2dx_1dx_2\sum_{K,H}\frac{K}{(2K+1)^3}\frac{1}{H^2}\end{eqnarray*}
where $\Delta^2$ denote the simplexe $x_1+x_2\leq 1$ in $\R_+^2$, and \[\int_{\Delta^2}x_1^2x_2dx_1dx_2=\frac{2!}{5!}.\]
Note that \[\sum_{K\geq 0,H\geq 1}\frac{K}{(2K+1)^3}\frac{1}{H^2}=S_1(2)\zeta(2)=\cfrac{1}{16}(6(\zeta(2))^2-7\zeta(2)\zeta(3))\]
with $S_1(m)$ defined on the proof of (\ref{eq:combi1}).
Similarly, we obtain that \[\sum_{W_1(2K+2)+W_2H_3\leq 2N}W_1^2W_2K=(2N)^5\frac{2!}{5!}\frac{1}{8}((\zeta(2))^2-\zeta(2)\zeta(3)),\]
which gives the result.

\end{proof}


\bibliographystyle{amsalpha}

\end{document}